\documentclass[11pt,reqno]{amsart}

\usepackage{amsmath, amsfonts, amsthm, amssymb}

\newtheorem{Theorem}{Theorem}[section]
\newtheorem{Lemma}{Lemma}[section]
\newtheorem{Proposition}{Proposition}[section]

\theoremstyle{definition}

\theoremstyle{remark}
\newtheorem{Remark}{Remark}[section]

\numberwithin{equation}{section}

\renewcommand{\u}{{\bf u}}

\newcommand{\R}{{\mathbb R}}
\newcommand{\Dv}{{\rm div}}

\newcommand{\m}{{\bf m}}

\def\f{\frac}
\renewcommand{\O}{\Omega}

\def\D{\Delta }

\def\hf1{^\f{1}{1-\xi^2}}

\def\be{\begin{equation}}
\def\en{\end{equation}}
\def\bs{\begin{split}}
\def\es{\end{split}}

\renewcommand{\v}{{\bf v}}

\author{Alexis F. Vasseur}
\address{Department of Mathematics,
The University of Texas at Austin.}
\email{vasseur@math.utexas.edu}
\author{Cheng Yu}
\address{Department of Mathematics,
The University of Texas at Austin.}
\email{yucheng@math.utexas.edu}
\title[Global solutions to the Navier-Stokes Equations]
{Existence of Global Weak Solutions for 3D Degenerate Compressible Navier-Stokes Equations}
\subjclass[2010]{35Q35, 76N10}
\keywords{Global weak solutions, compressible Navier-Stokes equations, shallow water equations, vacuum, degenerate viscosity.}

\date{\today}

\begin{document}
\begin{abstract}
 In this paper, we prove the existence of global weak solutions for  3D compressible Navier-Stokes equations with degenerate viscosity. The method is based on the Bresch and Desjardins entropy conservation \cite{BD}.  The main contribution of this paper is to derive  the Mellet-Vasseur type inequality \cite{MV} for the weak solutions, even if it is  not verified by the first level of approximation.  This provides existence of global solutions in time, for
 the compressible Navier-Stokes equations, for any $\gamma>1$ in two dimensional space and for $1<\gamma<3$ in three dimensional space, with large initial data possibly
  vanishing on the  vacuum. This solves an open problem proposed by Lions in \cite{Lions}.
\end{abstract}

\maketitle

\section{Introduction}
The existence of global weak solutions of compressible Navier-Stokes equations with degenerate viscosity has been a long standing open problem. The objective of this current paper is to establish the existence of global weak solutions to the following 3D compressible Navier-Stokes equations:
\begin{equation}
\label{NS equation}
\begin{split}
&\rho_t+\Dv(\rho\u)=0\\
&(\rho\u)_t+\Dv(\rho\u\otimes\u)+\nabla P-\Dv(\rho\mathbb{D}\u)=0,
\end{split}
\end{equation}
with initial data
\begin{equation}
\label{initial data}
\rho|_{t=0}=\rho_0(x),\;\;\;\;\;\rho\u|_{t=0}=\m_0(x).
\end{equation}
where $P=\rho^{\gamma},\,\gamma>1,$ denotes the pressure, $\rho$ is the density of fluid, $\u$ stands for the velocity of fluid, $\mathbb{D}\u=\frac{1}{2}[\nabla\u+\nabla^T\u]$ is the strain tensor.
For the sake of simplicity we will consider the case of bounded domain with periodic boundary conditions, namely $\O=\mathbb{T}^3$.
\vskip0.3cm
In the case $\gamma=2$ in two dimensional space, this corresponds to the shallow water equations, where $\rho(t,x)$ stands for the height of the water at position $x$, and time $t$, and $\u(t,x)$ is the 2D velocity at the same position, and same time. In this case, the physical viscosity was formally derived  as in (\ref{NS equation}) (see Gent \cite{G}). In this context, the global existence of weak solutions to equations (\ref{NS equation}) is proposed as an open   problem by Lions  in  \cite{Lions}. A careful derivation of the shallow water equations with the following viscosity term
$$2\Dv(\rho\mathbb{D}\u)+2\nabla(\rho\Dv\u)$$
can be found in the recent work by Marche \cite{FA}. Bresch-Noble \cite{BN,BN2} provided the mathematical derivation of viscous shallow-water equations with the above viscosity. However, this viscosity cannot be  covered by the  BD entropy.

\vskip0.3cm

Compared with the incompressible flows, dealing with the vacuum
is a very challenging problem in the study of the compressible flows.  Kazhikhov and Shelukhin \cite{KS} established the first existence result
 on the compressible Navier-Stokes equations in one dimensional space. Due to the difficulty from the vacuum, the initial density should be bounded away from zero in their work.
 It has been extended by Serre \cite{S} and Hoff \cite{Hoff87} for the discontinuous initial data, and by Mellet-Vasseur \cite{MV2} in the case of density dependent viscosity coefficient, see also in spherically symmetric case \cite{DNV00,DNV01,GJX}.  For the multidimensional case,
Matsumura and Nishida \cite{MN79,MN80,MN83} first established the global existence with the small initial data, and later by Hoff \cite{H95JDE,H95,H97} for discontinuous initial data.
To remove the difficulty from the vacuum,
 Lions in \cite{Lions} introduced the concept of  renormalized solutions to establish the global existence of weak solutions for $\gamma>\frac{9}{5}$ concerning large initial
  data that may vanish, and then Feireisl-Novotn\'{y}-Petzeltov\'{a} \cite{FNP} and Feireisl \cite{F04} extended the existence results to $\gamma>\frac{3}{2}$, and even to Navier-Stokes-Fourier system.
   In all above works, the viscosity coefficients were assumed to be fixed positive numbers.
   This is important to control the gradient of the velocity,  in the context of solutions close to an equilibrium, a breakthrough was obtained by Danchin \cite{D, D2}.
 However, the regularity and the uniqueness of the weak solutions for large data remain largely open for the compressible Navier-Stokes equations, even as in two dimensional space, see Vaigant-Kazhikhov \cite{VK} (see also Germain \cite{Germain}, and Haspot \cite{Haspot}, where criteria for regularity or uniqueness are proposed).
\vskip0.3cm

 The problem becomes even more challenging when the viscosity coefficients depend on the density. Indeed, the Navier-Stokes equations \eqref{generality equation} is highly
 degenerated at the vacuum because the velocity cannot even be defined when the density vanishes. It is very difficult to deduce any estimate of the gradient on
 the velocity field due to the vacuum. This is  the essential difference from the compressible Navier-Stokes equations with the non-density dependent viscosity coefficients.
  The first tool of handling this difficulty is due to Bresch, Desjardins and Lin, see \cite{BDL}, where the authors developed a new mathematical entropy to show the
   structure of the diffusion terms providing some regularity for the density. An early version of this entropy can be found in 1D for constant viscosity in \cite{Sh, V}. The result was later extended for the case
with an additional quadratic friction term $r\rho|\u|\u$,
refer to Bresch-Desjardins \cite{BD,BD2006} and the recent results by Bresch-Desjardins-Zatorska \cite{BDZ} and by Zatorska \cite{Z}.  Unfortunately, those bounds are not enough to treat the compressible Navier-Stokes equations without additional control on the vacuum, as the introduction of   capillarity,  friction, or cold pressure.
\vskip0.3cm

The primary obstacle to prove the compactness of the solutions to \eqref{generality equation} is the lack of  strong convergence for $\sqrt{\rho}\u$ in $L^2$. We cannot pass to the limit in the term $\rho\u\otimes\u$ without
 the strong convergence of $\sqrt{\rho}\u$ in $L^2$. This is an other essential difference with  the case of non-density dependent viscosity.
To solve this problem, a new estimate is established in Mellet-Vasseur \cite{MV},  providing a $L^{\infty}(0,T;L\log L(\O))$ control on  $\rho|\u|^2$.  This new estimate  provides the weak stability of  smooth solutions of \eqref{generality equation}.
\vskip0.3cm

The classical way to construct global weak solutions of \eqref{generality equation} would consist in
constructing  smooth approximation solutions, verifying the   priori estimates, including the Bresch-Desjardins entropy, and the Mellet-Vasseur inequality. However, those extra estimates impose a lot of structure on the approximating system. Up to now, no such approximation scheme has been discovered. In \cite{BD, BD2006}, Bresch and Desjardins propose a very nice construction of approximations, controlling both the usual energy and BD entropy.  This allows the construction of weak solutions, when additional terms -as drag terms, or cold pressure, for instance- are added. Note that their result holds true even in dimension 3. However, their construction does not provide the control of the $\rho \u$ in
 $L^{\infty}(0,T;L\log L(\O))$.
\vskip0.3cm

The objective of our current work is to investigate the issue of existence of solutions for the compressible Navier-Stokes equations \eqref{NS equation} with large initial data in 3D. Jungel \cite{J} studied the compressible
Navier-Stokes equations with the Bohm potential $\kappa\rho\left(\frac{\D\sqrt{\rho}}{\sqrt{\rho}}\right)$, and obtained the existence of a particular weak solution. Moreover, he deduced an estimate of $\nabla\rho^{\frac{1}{4}}$ in $L^4((0,T)\times\O),$ which is very useful in this current paper.
In \cite{GV}, Gisclon and Lacroix-Violet showed the existence of usual weak solutions for the compressible quantum Navier-Stokes equations with the addition of a cold pressure.
 Independently, we proved the existence of weak solutions to the compressible  quantum Navier-Stokes equations with damping terms, see \cite{VY-1}. This result is very similar to \cite{GV}.
  Actually, it is written in \cite{GV} that they can handle in a similar way the case with the drag force.
  Unfortunately, the case with the cold pressure is not suitable for our purpose.
\vskip0.3cm

 Building up from the result \cite{VY-1} (a variant of \cite{GV}),  we establish the logarithmic estimate for the weak solutions similar to \cite{MV}. For this, we first derive a ``renormalized" estimate on $\rho \varphi(|\u|)$, for $\varphi$ nice enough, for solutions of \cite{VY-1} with the additional drag forces. It is showed to be independent on the strength of those drag forces, allowing to pass into the limit when those forces vanish. Since this estimate cannot be derived from the approximation scheme of \cite{VY-1}, it has to be carefully derived on weak solutions. After passing into the limit $\kappa$ goes to 0, we can recover the logarithmic estimate, taking a suitable function $\varphi$. This is reminiscent to showing the conservation of the energy for weak solutions to incompressible Navier-Stokes equations. This conservation is true for smooth solutions. However, it is a long standing open problem, whether Leray-Hopf weak solutions are also conserving energy.
 \vskip0.3cm
Equation (\ref{NS equation}) can be seen as a particular case of the following Navier-Stokes
\begin{equation}
\label{generality equation}
\begin{split}
&\rho_t+\Dv(\rho\u)=0\\
&(\rho\u)_t+\Dv(\rho\u\otimes\u)+\nabla P-\Dv(\mu(\rho)\mathbb{D}\u)-\nabla(\lambda(\rho)\Dv\u)=0,
\end{split}
\end{equation}
where the  viscosity coefficients $\mu(\rho)$ and $\lambda(\rho)$ depend on the density, and  may vanish on the vacuum. When the coefficients verify the following condition:
$$\lambda(\rho)=2\rho\mu^{'}(\rho)-2\mu(\rho)$$
the system still formally verifies the BD estimates.  However, the construction of Bresch and Desjardins in \cite{BD2006} is more subtle in this case. Up to now, construction of weak solutions are known, only verifying a fixed combination of the classical energy and BD entropy (see \cite{BDZ}) in the case with additional terms. Those solutions verify the decrease of this so-called $\kappa$-entropy\footnote{Note that $\kappa$ here is not related to the $\kappa$ term in (\ref{NSK}).}, but not the decrease of  Energy and BD entropy  by themselves.
The extension of our result,  in this context,  is considered in \cite{VY}.
\vskip0.3cm

The basic energy inequality associated to \eqref{NS equation} reads as
\begin{equation}
\label{energy inequality for NS}
E(t)+\int_0^T\int_{\O}\rho|\mathbb{D}\u|^2\,dx\,dt\leq E_0,
\end{equation}
where $$E(t)=E(\rho,\u)(t)=\int_{\O}\left(\frac{1}{2}\rho|\u|^2+\frac{1}{\gamma-1}\rho^{\gamma}\right)\,dx,$$
and  $$E_0=E(\rho,\u)(0)=\int_{\O}\left(\frac{1}{2}\rho_0|\u_0|^2+\frac{1}{\gamma-1}\rho_0^{\gamma}\right)\,dx.$$
Remark that those a priori estimates are not enough to show the stability of the solutions of \eqref{NS equation}, in particular,
 for the compactness of $\rho^{\gamma}.$ Fortunately, a particular mathematical structure was found in \cite{BD,BDL}, which yields the bound of $\nabla\rho^{\frac{\gamma}{2}}$ in $L^{2}(0,T;L^2(\O)) $. More precisely,
we have the following Bresch-Desjardins entropy
\begin{equation*}
\begin{split}
&\int_{\O}\left(\frac{1}{2}\rho|\u+\nabla\ln\rho|^2+\frac{\rho^{\gamma}}{\gamma-1}\right)\,dx+\int_0^T\int_{\O}|\nabla\rho^{\frac{\gamma}{2}}|^2\,dx\,dt
\\&+\int_0^T\int_{\O}\rho|\nabla\u-\nabla^T\u|^2\,dx\,dt
\leq\int_{\O}\left(\rho_0|\u_0|^2+|\nabla\sqrt{\rho_0}|^2+\frac{\rho_0^{\gamma}}{\gamma-1}\right)\,dx.
\end{split}
\end{equation*}

Thus, the initial data should be given in such way that
\begin{equation}
\begin{split}
\label{initial condition}
&\rho_0\in L^{\gamma}(\O),\;\;\;\rho_0\geq 0,\;\;\; \nabla\sqrt{\rho_0}\in L^2(\O),\\
&\m_0\in L^1(\O),\;\;\m_0=0\;\;\text{ if } \;\rho_0=0,\;\;\frac{|\m_0|^2}{\rho_0}\in L^1(\O).
\end{split}
\end{equation}
\begin{Remark}
The initial condition $\nabla\sqrt{\rho_0}\in L^2(\O)$ is from the Bresch-Desjardins entropy.
\end{Remark}

The primary obstacle to prove the compactness of the solutions to \eqref{NSK} with $r_0=r_1=0$ is the lack of  strong convergence for $\sqrt{\rho}\u$ in $L^2$.
Jungel \cite{J} proved the existence of a particular weak solutions with test function $\rho\varphi,$ which was used in \cite{BDL}. The main idea of his paper is to rewrite quantum Navier-Stokes equations as a viscous quantum Euler system by means of the effective velocity. In \cite{J}, he also proved inequality \eqref{J inequality for weak solutions} which is crucial to get a key lemma in this current paper. Motivated by the works of \cite{BD, BDL, J}, we proved the existence of weak solutions to \eqref{NSK} and the inequality \eqref{J inequality for weak solutions}, see \cite{VY-1}. The advantage of $r_0$ and $r_1$ terms is that there is a compactness $\rho\u\otimes\u$ in $L^1$ and the strong convergence of $\sqrt{\rho}\u$ in $L^2.$
In particular, we need to recall the following existence result in \cite{VY-1}.
\begin{Proposition}
\label{pro weak solutions}

For any $\kappa\geq 0$,
there exists a global weak solution to the following system
\begin{equation}
\begin{split}
\label{NSK}
&\rho_t+\Dv(\rho\u)=0,
\\ &(\rho\u)_t+\Dv(\rho\u\otimes\u)+\nabla\rho^{\gamma}-\Dv(\rho\mathbb{D}\u)=-r_0\u-r_1\rho|\u|^2\u+\kappa\rho\nabla(\frac{\D\sqrt{\rho}}{\sqrt{\rho}}),
\end{split}
\end{equation}
with the initial data \eqref{initial data} and satisfying \eqref{initial condition} and $-r_0\int_{\O}\log_{-}\rho_0\,dx<\infty$. In particular, we have the energy inequality
\begin{equation}
\label{energy inequality for approximation}
E(t)+\int_0^T\int_{\O}\rho|\mathbb{D}\u|^2\,dx\,dt+r_0\int_0^T\int_{\O}|\u|^2\,dx\,dt+r_1\int_0^T\int_{\O}\rho|\u|^4\,dx\,dt\leq E_0,
\end{equation}
where $$E(t)=E(\rho,\u)(t)=\int_{\O}\left(\frac{1}{2}\rho|\u|^2+\frac{1}{\gamma-1}\rho^{\gamma}+\frac{\kappa}{2}|\nabla\sqrt{\rho}|^2\right)\,dx,$$
and  $$E_0=E(\rho,\u)(0)=\int_{\O}\left(\frac{1}{2}\rho_0|\u_0|^2+\frac{1}{\gamma-1}\rho_0^{\gamma}+\frac{\kappa}{2}|\nabla\sqrt{\rho_0}|^2\right)\,dx;$$
and the BD-entropy
\begin{equation}
\begin{split}
\label{BD entropy for approximation}
&\int_{\O}\left(\frac{1}{2}\rho|\u+\nabla\ln\rho|^2+\frac{\rho^{\gamma}}{\gamma-1}+\frac{\kappa}{2}|\nabla\sqrt{\rho}|^2-r_0\log\rho\right)\,dx+\int_0^T\int_{\O}|\nabla\rho^{\frac{\gamma}{2}}|^2\,dx\,dt
\\&+\int_0^T\int_{\O}\rho|\nabla\u-\nabla^T\u|^2\,dx\,dt+\kappa\int_0^T\int_{\O}\rho|\nabla^2\log\rho|^2\,dx\,dt
\\&\leq2\int_{\O}\left(\rho_0|\u_0|^2+|\nabla\sqrt{\rho_0}|^2+\frac{\rho_0^{\gamma}}{\gamma-1}+\frac{\kappa}{2}|\nabla\sqrt{\rho_0}|^2-r_0\log_{-}\rho_0\right)\,dx+2E_0,
\end{split}
\end{equation}
where $\log_{-}g=\log\min(g,1);$
 the following inequality for any weak solution $(\rho,\u)$
\begin{equation}
\label{J inequality for weak solutions}
\kappa^{\frac{1}{2}}\|\sqrt{\rho}\|_{L^2(0,T;H^2(\O))}+\kappa^{\frac{1}{4}}\|\nabla\rho^{\frac{1}{4}}\|_{L^4(0,T;L^{4}(\O))}\leq C,
\end{equation}
where $C$  only depends on the initial data.

 Moreover, the weak solution $(\rho,\u)$ has the following properties
\begin{equation}
\begin{split}
\label{property}
&\rho\u\in C([0,T];L^{\frac{3}{2}}_{weak}(\O)),\quad(\sqrt{\rho})_t\in L^2((0,T)\times\O);
\end{split}
\end{equation}

If we use $(\rho_{\kappa},\u_{\kappa})$ to denote the weak solution for $\kappa>0$, then
\begin{equation}
\begin{split}
\label{property-2}
&\sqrt{\rho_{\kappa}}\u_{\kappa}\to \sqrt{\rho}\u\,\,\,\,\text{strongly in } L^2((0,T)\times\O),\;\;\text{as } \;\kappa\to0,
\end{split}
\end{equation}
where $(\rho,\u)$ in \eqref{property-2} is a weak solution to \eqref{NSK} and \eqref{initial data} for $\kappa=0.$
\end{Proposition}
\begin{Remark}
The energy inequality \eqref{energy inequality for approximation} yields the following estimates
\begin{equation}
\begin{split}
\label{a priori estimate from energy}
&\|\sqrt{\rho}\u\|_{L^{\infty}(0,T;L^2(\O))}\leq E_0<\infty,\\
&\|\rho\|_{L^{\infty}(0,T;L^{\gamma}(\O))}\leq E_0<\infty,\\
&\|\sqrt{\kappa}\nabla\sqrt{\rho}\|_{L^{\infty}(0,T;L^2(\O))}\leq E_0<\infty,\\
&\|\sqrt{\rho}\mathbb{D}\u\|_{L^{2}(0,T;L^2(\O))}\leq E_0<\infty,\\
&\|\sqrt{r_0}\u\|_{L^{2}(0,T;L^{2}(\O))}\leq E_0<\infty,\\
&\|\sqrt[4]{r_1\rho}\u\|_{L^4(0,T;L^4(\O))}\leq E_0<\infty.
\end{split}
\end{equation}
The BD entropy \eqref{BD entropy for approximation} yields the following bounds on the density $\rho$:
\begin{equation}
\label{estimate of sqrt density}
\|\nabla\sqrt{\rho}\|_{L^{\infty}(0,T;L^{2}(\O))}\leq C<\infty,
\end{equation}
\begin{equation}
\label{estimate of second derivative density}
\|\sqrt{\kappa\rho}\nabla^2\log\rho\|_{L^2(0,T;L^2(\O))}\leq C<\infty,
\end{equation}
\begin{equation}
\label{estimate for presure}
\|\nabla\rho^{\frac{\gamma}{2}}\|_{L^{2}(0,T;L^{2}(\O))}\leq C<\infty,
\end{equation}
and
\begin{equation}
\label{estimate on u}
\|\sqrt{\rho}\nabla\u\|^2_{L^2(0,T;L^2(\O))}\leq\int_{\O}\left(\rho_0|\u_0|^2+\frac{\rho^{\gamma}_0}{\gamma-1}+|\nabla\sqrt{\rho_0}|^2-r_0\log_{-}\rho_0\right)\,dx+2 E_0<\infty,
\end{equation}
where $C$ is bounded by the initial data, uniformly on $r_0$, $r_1$ and $\kappa$.\\
In fact, \eqref{estimate of sqrt density} yields
\begin{equation}
\label{estimate on density in Lp}
\sqrt{\rho}\in L^{\infty}(0,T;L^6(\O)),
\end{equation}
in three dimensional space.
\end{Remark}
\begin{Remark} Inequality
\eqref{J inequality for weak solutions} is a consequence of the bound on \eqref{estimate of second derivative density}. This was used already in \cite{J}. The estimate for the full system \eqref{NSK} is proved in \cite{VY-1}.
\end{Remark}

\begin{Remark}
The existence result of \cite{BD} contained the case
 with $\kappa=0$,  which can be obtained as the limit when $\kappa>0$ goes to 0 in (\ref{NSK}), by  standard compactness analysis.
\end{Remark}
\begin{Remark}
\label{remark on weak formulation} The weak formulation reads as
\begin{equation}
\begin{split}
\label{weak formulation of NSK}
&\int_{\O}\rho\u\cdot\psi\,dx|_{t=0}^{t=T}-\int_{0}^{T}\int_{\O}\rho\u\psi_t\,dx\,dt
-\int_{0}^{T}\int_{\O}\rho\u\otimes\u:\nabla \psi\,dx\,dt\\&-\int_{0}^{T}\int_{\O}\rho^{\gamma}\Dv\psi\,dx\,dt
-\int_0^{T}\int_{\O}\rho\mathbb{D}\u:\nabla\psi\,dx\,dt
\\& =-r_0\int_{0}^{T}\int_{\O}\u\psi\,dx\,dt-r_1\int_0^T\int_{\O}\rho|\u|^2\u\psi\,dx\,dt-2\kappa\int_0^T\int_{\O}\D\sqrt{\rho}\nabla\sqrt{\rho}\psi\;dx\;dt
\\&-\kappa\int_0^T\int_{\O}\D\sqrt{\rho}\sqrt{\rho}\Dv\psi\,dx\,dt.
\end{split}
\end{equation}
for any test function $\psi.$
\end{Remark}

Our first main result reads as follows:
\begin{Theorem}
\label{MV inequality}
For any  $\delta\in(0,2)$, there exists a constant $C$ depending only on $\delta$, such that the following holds true.
There exists  a weak solution $(\rho,\u)$ to \eqref{NSK} with $\kappa=0$  verifying all the properties of  Proposition \ref{pro weak solutions}, and  satisfying the following Mellet-Vasseur type inequality for every $T>0$, and almost every $t<T$:
\begin{equation*}
\begin{split}
&\int_{\O} \rho(t,x)(1+|\u(t,x)|^2)\ln(1+|\u(t,x)|^2)\,dx
\\&\leq  \int_{\O}\rho_0(1+|\u_0|^2)\ln(1+|\u_0|^2)\,dx+8\int_{\O}\left(\rho_0|\u_0|^2+\frac{\rho^{\gamma}_0}{\gamma-1}+|\nabla\sqrt{\rho_0}|^2-r_0\log_{-}\rho_0\right)\,dx
\\&+16 E_0
+C\int_0^T \left(\int_{\O}(\rho^{2\gamma-1-\frac{\delta}{2}})^{\frac{2}{2-\delta}}\right)^{\frac{2-\delta}{2}}\left(\int_{\O}\rho(2+\ln(1+|\u|^2))^{\frac{2}{\delta}}\,dx\right)^{\frac{\delta}{2}}\,dt,
\end{split}
\end{equation*}
where $\gamma>1$ in two dimensional space and $1<\gamma<3$ in three dimensional space.
\end{Theorem}
\begin{Remark} The right hand side of the above inequality can be bounded by the initial data. In particular, it does not depend on $r_0$ and $r_1$.
This theorem will yield the strong convergence of $\sqrt{\rho}\u$ in space $L^2(0,T;\O)$ when $r_0$, $r_1$ converge to $0.$ It will be the key tool of obtaining the existence of weak solutions, in \cite{MV}.
\end{Remark}

We define the weak solution $(\rho,\u)$ to the initial value problem \eqref{NS equation} in the following sense: for any $t\in[0,T]$,
\begin{itemize}
\item  \eqref{initial data} holds in $\mathcal{D'}(\O)$,
\item \eqref{energy inequality for NS} holds for almost every $t\in[0,T]$,
\item \eqref{NS equation} holds in $\mathcal{D'}((0,T)\times\O))$ and the following is satisfied\\
$\rho\geq 0, \quad \rho \in L^{\infty}([0,T];L^{\gamma}(\O)),$
\\$\rho(1+|\u|^2)\ln (1+|\u|^2)\in L^{\infty}(0,T;L^1(\O)),$
\\$ \nabla\rho^{\frac{\gamma}{2}}\in L^2(0,T;L^2(\O)),\quad\nabla\sqrt{\rho}\in L^{\infty}(0,T;L^2(\O)),$\\
           $  \sqrt{\rho}\u \in L^{\infty}(0,T;L^2(\O)),\quad\sqrt{\rho}\nabla\u\in L^2(0,T;L^2(\O)).$
 \end{itemize}
 \begin{Remark}
The regularity $\nabla\sqrt{\rho}\in L^{\infty}(0,T;L^2(\O))$ and $\nabla\rho^{\frac{\gamma}{2}}\in L^2(0,T;L^2(\O))$ are
 from the Bresch-Desjardins entropy.
\end{Remark}
As a sequence of Theorem \ref{MV inequality}, our second main result reads as follows:
\begin{Theorem}
\label{main result}
Let $(\rho_0,\,\m_0)$ satisfy \eqref{initial condition} and
\begin{equation*}
\int_{\O}\rho_0(1+|\u_0|^2)\ln(1+|\u_0|^2)\,dx<\infty.
\end{equation*}
 Then,
for $\gamma>1$ in two dimensional space and $1<\gamma<3$ in three dimensional space, and any $T>0$, there exists a weak solution of \eqref{NS equation}-\eqref{initial data} on $(0,T).$
\end{Theorem}
\vskip0.3cm
We cannot obtained directly the estimate of Theorem \ref{MV inequality}  from (\ref{NSK}) with $\kappa=0$, because we do not have enough regularity on the solutions. But, the estimate is not true for the solutions of (\ref{NSK}) for $\kappa>0$. The idea is to obtain a control on
$$
\int_{\Omega}\rho(t,x)\varphi_n(\phi(\rho)\u(t,x))\,dx
$$
at the level $\kappa>0$, for a $\varphi_n$, suitable bounded approximation of $(1+|\u|^2)\ln(1+|\u|^2)$, and a suitable cut-off function $\phi$ of $\rho$, controlling both the large and small $\rho$.
The first step  (see section 2) consists in showing that we can control (uniformly with respect to $\kappa$) this quantity, for any {\it{weak}} solutions of  (\ref{NSK}) with $\kappa>0$. This has to be done in several steps, taking into account the minimal regularity of the solutions, the weak control of the solutions close to the vacuum, and the extra capillarity higher order terms. In the limit $\kappa$ goes to zero, the cut-off function $\phi$ has to converge to  one in a special rate associated to $\kappa$ (see section 3 and 4).  This provides,  for any  weak limit of (\ref{NSK}) obtained by limit $\kappa$ converges to 0,  a (uniform in $n$, $r_0$, and $r_1$) bound,  to:
$$
\int_{\Omega}\rho(t,x)\varphi_n(\phi(\rho)\u(t,x))\,dx.
$$
Note that the bound is not uniform in $n$, for $\kappa$ fixed. However, it becomes uniform in $n$ at the limit $\kappa$ converges to 0. In section 5, we pass into the limit $n$ goes to infinity, obtaining  a uniform bound with respect to $r_0$ and $r_1$ of
$$
\int_{\Omega}\rho(t,x)(1+|\u(t,x)|^2)\ln (1+|\u(t,x)|^2)\,dx.
$$
Section 6 is devoted to the limit $r_1$ and $r_0$ converges to 0. The uniform estimate above provides the strong convergence of $\sqrt{\rho}\u$ needed to obtain the existence of global weak solutions to
\eqref{NS equation} with large initial data.

\bigskip

\bigskip

\section{Approximation of the Mellet-Vasseur type inequality}
In this section, we construct an approximation of the Mellet-Vasseur type inequality for any weak solutions at the following level of approximation system
\begin{equation}
\label{last level}
\begin{split}
&\rho_t+\Dv(\rho\u)=0,
\\ &(\rho\u)_t+\Dv(\rho\u\otimes\u)+\nabla\rho^{\gamma}-\Dv(\rho\mathbb{D}\u)=-r_0\u-r_1\rho|\u|^2\u+\kappa\rho\nabla(\frac{\Delta\sqrt{\rho}}{\sqrt{\rho}}),
\end{split}
\end{equation}
with the initial data \eqref{initial condition}, verifying in addition that $\rho_0\geq \frac{1}{m_0}$ for $m_0>0$ and $\sqrt{\rho_0}\u_0\in L^{\infty}(\O).$ This restriction on the initial data will be
useful later to get the strong convergence of $\sqrt{\rho}\u$ when $t$ converges to $0$. This restriction will be cancel at the very end, (see section 6).


In the same line of Bresch-Desjardins \cite{BD,BD2006,J}, we constructed the weak solutions to the system \eqref{NSK} for any $\kappa\geq0$  by the natural energy estimates and the Bresch-Desjardins entropy, see \cite{VY-1}.
 The term $r_0\u$ turns out to be essential to show the strong convergence of $\sqrt{\rho}\u$ in $L^2(0,T;L^2(\O))$.
 Unfortunately, it is not enough to ensure the strong convergence of $\sqrt{\rho}\u$ in $L^2(0,T;L^2(\O))$ when $r_0$ and $r_1$ vanish.
\vskip0.3cm

We define two $C^\infty$, nonnegative  cut-off functions $\phi_m$ and $\phi_K$ as follows.
\begin{equation}
\label{cut function}
\phi_m(\rho) =1\,\text{ for any } \rho>\frac{1}{m},\;\; \;\phi_m(\rho) =0\,\text{ for any } \rho<\frac{1}{2m},
\end{equation}
where $m>0$ is any real number, and $|\phi'_m|\leq 2m$;\newline
and $\phi_K(\rho)\in C^{\infty}(\R)$ is  a  nonnegative function such that
\begin{equation}
\label{cut function}
\phi_K(\rho) =1\,\text{ for any } \rho<K,\;\; \;\phi_K(\rho) =0\,\text{ for any } \rho>2K,
\end{equation}
where $K>0$ is any real number, and $|\phi'_K|\leq \frac{2}{K}$.
\vskip0.3cm
We define  $\v=\phi(\rho)\u$, and $\phi(\rho)=\phi_m(\rho)\phi_K(\rho).$ The following Lemma will be very useful to construct the approximation of the Mellet-Vasseur type inequality. The structure of the $\kappa$ quantum term in \cite{J} is essential to get this lemma in 3D. It seems not possible to get it from the  Korteweg term of BD \cite{BD} in 3D.
\begin{Lemma}
\label{lemma of estimates for positive kappa} For any fixed $\kappa>0$, we have
\begin{equation*}
\|\nabla\v\|_{L^2(0,T;L^2(\O))}\leq C,
\end{equation*}
where the constant $C$ depend on $\kappa>0,$ $r_1$, $K$ and $m$;
and
\begin{equation*}
\rho_t\in L^4(0,T;L^{6/5}(\O))+L^2(0,T;L^{3/2}(\O)) \quad\text{ uniformly in } \kappa.
\end{equation*}
\end{Lemma}

\begin{proof} By \eqref{J inequality for weak solutions}, we have
$$\|\nabla\rho^{\frac{1}{4}}\|_{L^4(0,T;L^4(\O))}\leq C.$$
For $\v$, we have
\begin{equation*}
\nabla\v=\nabla(\phi(\rho)\u)=(\phi'(\rho)\nabla\rho)\u+\phi(\rho)\nabla\u,
\end{equation*}
and hence
\begin{equation*}
\begin{split}
&\|(\phi'(\rho)\nabla\rho)\u+\phi(\rho)\nabla\u\|_{L^2(0,T;L^2(\O))}
\\&\leq C\|\rho^{\frac{1}{4}}\u\nabla\rho^{\frac{1}{4}}\|_{L^2(0,T;L^2(\O))}+C\|\sqrt{\rho}\nabla\u\|_{L^2(0,T;L^2(\O))}
\\&\leq C\|\rho^{\frac{1}{4}}\u\|_{L^4(0,T;L^4(\O))}\|\nabla\rho^{\frac{1}{4}}\|_{L^4(0,T;L^4(\O))}+C\|\sqrt{\rho}\nabla\u\|_{L^2(0,T;L^2(\O))},
\end{split}
\end{equation*}
where we used the definition of the function $\phi(\rho)$. Indeed, there exists $C>0$ such that
\begin{equation*}
\left|\phi'(\rho)\sqrt{\rho}\right|+\left|\frac{\phi(\rho)}{\sqrt{\rho}}\right|\leq C
\end{equation*}
for any $\rho>0.$\\
For $\rho_t$, we have
\begin{equation*}
\begin{split}
&\rho_t=-\nabla\rho\cdot\u-\rho\Dv\u
\\&=-2\nabla\sqrt{\rho}\cdot\rho^{\frac{1}{4}}\u\rho^{\frac{1}{4}}-\sqrt{\rho}\sqrt{\rho}\Dv\u=S_1+S_2.
\end{split}
\end{equation*}

Thanks to \eqref{a priori estimate from energy}, \eqref{estimate of sqrt density} and \eqref{estimate on density in Lp}, we have $$S_1\in L^4(0,T;L^{r}(\O)) \;\;\text{ for } 1\leq r\leq\frac{6}{5}.$$
By \eqref{a priori estimate from energy} and \eqref{estimate on density in Lp}, we have
$$S_2\in L^2(0,T;L^s(\O))\;\;\text{for } 1\leq s\leq \frac{3}{2}.$$
Thus, we have
\begin{equation*}
\rho_t\in L^4(0,T;L^{r}(\O))+L^2(0,T;L^s(\O)).
\end{equation*}

\end{proof}
We introduce a new $C^1(\R^3)$, nonnegative  cut-off function $\varphi_n$:
\begin{equation}
\label{relation}
\varphi_n(\u)=\tilde{\varphi}_n(|\u|^2),
\end{equation}
where  $\tilde{\varphi}_n$ is given on $\R^+$ by
\begin{equation}
\label{second deritative-1}
\tilde{\varphi}_n''(y)\begin{cases}= \frac{1}{1+y} \;\;\;\;\;\quad\quad\quad\quad\text{ if }0\leq y\leq n,
\\ =-\frac{1}{1+y}\;\;\;\;\;\;\;\;\;\text{ if } n< y<C_n
\\ =0\,\;\;\;\;\quad\quad\;\quad\quad\quad\text{ if }  y\geq C_n,
\end{cases}\end{equation}
with $\tilde{\varphi}'_n(0)=0$, $\tilde{\varphi}_n(0)=0,$ and $C_n=e(1+n)^2-1$.

Here we gather the properties of the function $\tilde{\varphi}_n$ in the following Lemma:
\begin{Lemma}
\label{Lemma for the cut function}
Let $\varphi_n$ and $\tilde{\varphi}_n$ be defined as above. Then they verify
\begin{itemize}
\item a.
For any $\u\in \R^3$, we have \begin{equation}
\label{second deritative}\varphi_n''(\u)=2\left(2\tilde{\varphi}_n''(|\u|^2)\u\otimes\u+\mathbf{I} \tilde{\varphi}_n'(|\u|^2)\right),
\end{equation}
where $\mathbf{I}$ is $3\times 3$ identity matrix.

\item b.
$|\tilde{\varphi}''_n(y)|\leq \frac{1}{1+y}$ for any $n>0$ and any $y\geq 0$.

\item c.
\begin{equation}
\label{first deritative}
\tilde{\varphi}_n'(y)\begin{cases} =1+\ln (1+y) \;\;\;\;\;\text{ if }0\leq y\leq n,
\\ =0\,\;\;\;\;\quad\quad\;\quad\quad\text{ if } y\geq C_n,
\\ \geq 0, \text{ and } \leq 1+\ln (1+y)\,\;\;\text{ if } n<y\leq C_n.
\end{cases}\end{equation}
In one word, $0\leq \tilde{\varphi}'_n\leq 1+\ln(1+y)$ for any $y\geq 0,$ and it is compactly supported.

\item  d. For any given $n>0,$ we have
\begin{equation}
\label{bound of second deritative of cut function}
|\varphi^{''}_n(\u)|\leq 6+2\ln(1+n)
\end{equation}
for any $\u\in \R^3.$

\item e. \begin{equation}
\label{definition of cutting big u}
\tilde{\varphi}_n(y)=\begin{cases} (1+y)\ln (1+y) \;\;\;\;\;\quad\quad\quad\quad\quad\text{ if }0\leq y< n,
\\2(1+\ln(1+n))y-(1+y)\ln(1+y)+2(\ln(1+n)-n), \text{ if } n\leq y\leq   C_n,
\\e(1+n)^2-2n-2\,\;\;\;\;\quad\quad\quad\quad\quad\;\text{ if } y \geq  C_n,
\end{cases}\end{equation}
$\tilde{\varphi}_n(y)$ is a nondecreasing function with respect to $y$ for any fixed $n$, and it is a nondecreasing function with respect to $n$ for any fixed $y$.
\begin{equation}
\label{almost everywhere convergence for cut function}
\tilde{\varphi}_n(y)\to (1+y)\ln(1+y)\,\,\text{a.e.}
\end{equation}
as $n\to\infty.$
\end{itemize}
\end{Lemma}

\begin{proof}
We prove each statement one by one as follows:

\begin{itemize}
\item a.  Thanks to \eqref{relation}, we have $\varphi'_n(\u)=2\tilde{\varphi}'_n(|\u|^2)\u$, and
\begin{equation*}
\label{second deritative}\varphi_n''(\u)=2\left(2\tilde{\varphi}_n''(|\u|^2)\u\otimes\u+\mathbf{I} \tilde{\varphi}_n'(|\u|^2)\right),
\end{equation*}
where $\mathbf{I}$ is $3\times 3$ identity matrix.
\item b. The statement of b. follows directly from  \eqref{second deritative-1}.

\item c.
Integrating \eqref{second deritative-1} with initial data  $\tilde{\varphi}'_n(0)=0,$ we have
 \begin{equation}
 \label{daoshu}
\tilde{\varphi}'_n(y)=\begin{cases} 1+\ln (1+y) \;\;\;\;\;\quad\quad\quad\quad\quad\text{ if }0\leq y< n,
\\1+2\ln(1+n)-\ln(1+y), \text{ if } n\leq y\leq  C_n
\\0\,\;\;\;\;\quad\quad\quad\quad\quad\;\text{ if } y \geq C_n,
\end{cases}\end{equation}

Since $$1+2\ln(1+n)-\ln(1+C_n)= 0,$$
thus, for any $y\geq 0$, we have $$\tilde{\varphi}'_n(y)\geq 0.$$
For any $n\leq y\leq C_n,$ we have
$$1+2\ln(1+n)-\ln(1+y)\leq 1+2\ln(1+y)-\ln(1+y)=1+\ln(1+y).$$
Thus, for any $y\geq 0$, we have
$$\tilde{\varphi}'_n(y)\leq 1+\ln(1+y).$$

\item d.
By a.-c., $$|\varphi^{''}_n(\u)|\leq 4|\tilde{\varphi}'_n||\u|^2+2|\tilde{\varphi}'_n|\leq 4\frac{|\u|^2}{1+|\u|^2}+2(1+\ln(1+n))\leq 6+2\ln(1+n).$$
\item f.
Integrating  \eqref{daoshu} with initial data $\tilde{\varphi}_n(0)=0$, it gives \eqref{definition of cutting big u}. Moreover,
thanks to c., $\tilde{\varphi}_n(y)$ is an increasing function with respect to $y$ for any fixed $n$.
We  have also that $\tilde{\varphi}_n(y)$ is a nondecreasing function with respect to $n$ for any fixed $y$.

\end{itemize}

\end{proof}

The first step of constructing the approximation of the Mellet-Vasseur type inequality is the following lemma:
\begin{Lemma}
For any weak solutions to \eqref{last level} constructed in Proposition \ref{pro weak solutions}, and any $\psi(t)\in \mathfrak{D}(-1,+\infty)$, we have
\begin{equation}
\label{the first level of MV inequality}
\begin{split}
-\int_0^T\int_{\O}\psi_t\rho\varphi_n(\v)\,dx\,dt&+\int_0^T\int_{\O}\psi(t)\varphi'_n(\v)F\,dx\,dt+\int_0^T\int_{\O}\psi(t)\mathbf{S}:\nabla(\varphi'_n(\v))\,dx\,dt
\\&=\int_{\O}\rho_0\varphi_n(\v_0)\psi(0)\,dx,
\end{split}
\end{equation}
where \begin{equation}
\begin{split}
\label{definition of SF}
&\mathbf{S}=\rho\phi(\rho)(\mathbb{D}\u+\kappa\frac{\Delta\sqrt{\rho}}{\sqrt{\rho}} \mathbb{I}),\;\;\text{ and }\\
&F=\rho^2\u\phi'(\rho)\Dv\u+2\rho^{\frac{\gamma}{2}}\nabla\rho^{\frac{\gamma}{2}}\phi(\rho)
+\rho\nabla\phi(\rho)\mathbb{D}\u+r_0\u\phi(\rho)
\\&+r_1\rho|\u|^2\u\phi(\rho)
+\kappa\sqrt{\rho}\nabla\phi(\rho)\Delta\sqrt{\rho}+2\kappa\phi(\rho)\nabla\sqrt{\rho}\Delta\sqrt{\rho},
\end{split}
\end{equation}
where $\mathbb{I}$ is an identical matrix.
\end{Lemma}
In this proof, $\kappa$, $m$ and $K$ are fixed. So the dependence of the constants appearing in this proof will not be specified.

Multiplying $\phi(\rho)$ on both sides of the second equation of \eqref{last level}, we have
\begin{equation*}
\begin{split}
&(\rho\v)_t-\rho\u\phi'(\rho)\rho_t+\Dv(\rho\u\otimes\v)-\rho\u\otimes\u\nabla\phi(\rho)+2\rho^{\frac{\gamma}{2}}\nabla\rho^{\frac{\gamma}{2}}\phi(\rho)
\\&-\Dv(\phi(\rho)\rho\mathbb{D}\u)
+\rho\nabla\phi(\rho)\mathbb{D}\u+r_0\u\phi(\rho)+r_1\rho|\u|^2\u\phi(\rho)-\kappa\nabla(\sqrt{\rho}\phi(\rho)\Delta\sqrt{\rho})
\\&+\kappa\sqrt{\rho}\nabla\phi(\rho)\Delta\sqrt{\rho}+2\kappa\phi(\rho)\nabla\sqrt{\rho}\Delta\sqrt{\rho}=0.
\end{split}
\end{equation*}
\begin{Remark}
Both $\nabla\sqrt{\rho}$ and $\rho_t$ are functions, so the above equality are justified by regularizing $\rho$ and passing into the limit.
\end{Remark}
 We can rewrite the above equation as follows
 \begin{equation}
 \label{new form}
 (\rho\v)_t+\Dv(\rho\u\otimes\v)-\Dv\mathbf{S}+F=0,
 \end{equation}
where $\mathbf{S}$ and $F$ are as in \eqref{definition of SF}, and here we used
\begin{equation*}
\begin{split}
\rho\u\phi'(\rho)\rho_t&+\rho\u\otimes\u\phi'(\rho)\nabla\rho=\rho\u\phi'(\rho)(\rho_t+\nabla\rho\cdot\u)
\\&=-\rho^2\u\phi'(\rho)\Dv\u.
\end{split}
\end{equation*}
We should remark that, thanks to \eqref{J inequality for weak solutions}, \eqref{a priori estimate from energy}-\eqref{estimate on u},
\begin{equation*}
\|F\|_{L^{\frac{4}{3}}(0,T;L^1(\O))}\leq C,\;\;\;\|\mathbf{S}\|_{L^2(0,T;L^2(\O))}\leq C,
\end{equation*}
since $\sqrt{\rho}\phi(\rho)$ and $\rho\phi(\rho)$ bounded. Those bounds depend on $K$ and $\kappa.$\\

We first introducing a test function $\psi(t)\in  \mathfrak{D}(0,+\infty)$. Essentially this function vanishes for $t$ close $t=0$. We will later extend the result for $\psi(t)\in  \mathfrak{D}(-1,+\infty)$.
We define a new function $\Phi=\overline{\psi(t)\varphi'_n(\overline{\v})}$, where $\overline{f(t,x)}=f*\eta_{k}(t,x)$, $k$ is a small enough number. Note that, since $\psi(t)$ is compactly supported in $(0,\infty).$ $\Phi$ is well defined on $(0,\infty)$ for $k$ small enough. We use it to test \eqref{new form} to have
\begin{equation*}
\int_0^T\int_{\O}\overline{\psi(t)\varphi'_n(\overline{\v})}[(\rho\v)_t+\Dv(\rho\u\otimes\v)-\Dv\mathbf{S}+F]\,dx\,dt=0,
\end{equation*}
which in turn gives us
\begin{equation}
\label{weak formulation of new form}
\int_0^T\int_{\O}\psi(t)\varphi'_n(\overline{\v})\overline{[(\rho\v)_t+\Dv(\rho\u\otimes\v)-\Dv\mathbf{S}+F]}\,dx\,dt=0.
\end{equation}

 The first term in \eqref{weak formulation of new form} can be calculated as follows
 \begin{equation}
 \begin{split}
 \label{the first term of new form}
&\int_0^T\int_{\O}\psi(t)\varphi'_n(\overline{\v})\overline{(\rho\v)_t}\,dx\,dt
\\&=\int_0^T\int_{\O}\psi(t)\varphi'_n(\overline{\v})(\rho\overline{\v})_t\,dx\,dt+\int_0^T\int_{\O}\psi(t)\varphi'_n(\overline{\v})[\overline{(\rho\v)_t}-(\rho\overline{\v})_t]\,dx\,dt
\\&=\int_0^T\int_{\O}\psi(t)\varphi'_n(\overline{\v})(\rho_t\overline{\v}+\rho\overline{\v}_t)\,dx\,dt+R_1
\\&=\int_0^T\int_{\O}\psi(t)\rho_t\varphi'_n(\overline{\v})\overline{\v}\,dx\,dt+\int_0^T\int_{\O}\psi(t)\rho\varphi_n(\overline{\v})_t\,dx\,dt+R_1,
\end{split}
\end{equation}
where $$R_1=\int_0^T\int_{\O}\psi(t)\varphi'_n(\overline{\v})[\overline{(\rho\v)_t}-(\rho\overline{\v})_t]\,dx\,dt.$$

Thanks to the first equation in \eqref{last level}, we can rewrite the second term in \eqref{weak formulation of new form} as follows
\begin{equation}
\begin{split}
\label{the second term of new form}
&\int_0^T\int_{\O}\psi(t)\varphi'_n(\overline{\v})\overline{\Dv(\rho\u\otimes\v)}\,dx\,dt\\
&=\int_0^T\int_{\O}\psi(t)\rho_t\varphi_n(\overline{\v})\,dx\,dt-\int_0^T\int_{\O}\psi(t)\rho_t\varphi'_n(\overline{\v})\overline{\v}+R_2,
\end{split}
\end{equation}
 and $$R_2=\int_0^T\int_{\O}\psi(t)\varphi'_n(\overline{\v})[\Dv(\rho\u\otimes\overline{\v})-\overline{\Dv(\rho\u\otimes\v)}].$$
By \eqref{weak formulation of new form}-\eqref{the second term of new form}, we have
\begin{equation}
\label{weak formulation with R1 R2 A}
\begin{split}
\int_0^T\int_{\O}\psi(t)(\rho\varphi_n(\overline{\v}))_t\,dx\,dt+R_1&+R_2-\int_0^T\int_{\O}\psi(t)\varphi'_n(\overline{\v})\overline{\Dv{\mathbf{S}}}\;dx\,dt
\\&+\int_0^T\int_{\O}\psi(t)\varphi'_n(\overline{\v})\overline{F}=0.
\end{split}
\end{equation}

Notice that $\overline{\v}$ converges to $\v$ almost everywhere and $$\rho\varphi_n(\overline{\v})\psi_{t}\to\rho\varphi_n(\v)\psi_{t}\;\;\text{in } L^1((0,T)\times\O).$$ So, up to a subsequence,
we have
\begin{equation}
\label{convegence first term}
\int_0^T\int_{\O}(\rho\varphi_n(\overline{\v}))\psi_t\,dx\,dt\to \int_0^T\int_{\O}(\rho\varphi_n(\v))\psi_t\,dx\,dt\;\;\text{as }\;k\to0.
\end{equation}
Since $\varphi'_n(\overline{\v})$ converges to $\varphi'_n(\v)$ almost everywhere,  and is uniformly bounded in $ L^{\infty}(0,T;\O),$
we have
\begin{equation}
\label{convergence F term}
\int_0^T\int_{\O}\psi(t)\varphi'_n(\overline{\v})\overline{F}\to \int_0^T\int_{\O}\psi(t)\varphi'_n(\v)F \;\;\text{as }\;k\to0.
\end{equation}
Noticing that $$\nabla\v\in L^2(0,T;L^2(\O)),$$
we have $$\overline{\nabla\v}\to \nabla\v\;\;\text{ strongly in } L^2(0,T;L^2(\O)).$$
Since $\overline{ \mathbf{S}}$ converges to $ \mathbf{S}$ strongly in $ L^2(0,T;L^2(\O)),$ and
$ \varphi''_n(\overline{\v})$ converges to $\varphi''_n(\v)$ almost everywhere  and uniformly bounded in $ L^{\infty}((0,T)\times\O),$
 we get
\begin{equation}
\label{S term}
\int_0^T\int_{\O}\psi(t)\varphi'_n(\overline{\v})\overline{\Dv{\mathbf{S}}}\;dx\,dt=-
\int_0^T\int_{\O}\psi(t)\overline{\mathbf{S}}:\nabla(\varphi'_n(\overline{\v}))\;dx\,dt,
\end{equation}
which converges to
\begin{equation}
\label{S term2}
-\int_0^T\int_{\O}\psi(t)\mathbf{S}:\nabla(\varphi'_n(\v))\;dx\,dt.
\end{equation}

To handle $R_1$ and $R_2$, we use the following lemma due to Lions, see \cite{Lions1}.
 \begin{Lemma}
 \label{Lions's lemma}
 Let $f\in W^{1,p}(\R^N),\,g\in L^{q}(\R^N)$ with $1\leq p,q\leq \infty$, and $\frac{1}{p}+\frac{1}{q}\leq 1$. Then, we have
 $$\|\Dv(fg)*w_{\varepsilon}-\Dv(f(g*w_{\varepsilon}))\|_{L^{r}(\R^N)}\leq C\|f\|_{W^{1,p}(\R^N)}\|g\|_{L^{q}(\R^N)}$$
 for some $C\geq 0$ independent of $\varepsilon$, $f$ and $g$, $r$ is determined by $\frac{1}{r}=\frac{1}{p}+\frac{1}{q}.$ In addition,
  $$\Dv(fg)*w_{\varepsilon}-\Dv(f(g*w_{\varepsilon}))\to0\;\;\text{ in }\,L^{r}(\R^N)$$
 as $\varepsilon \to 0$ if $r<\infty.$
 \end{Lemma}
This lemma includes the following statement.
 \begin{Lemma}
 \label{similar to Lions's lemma}
 Let $f_t\in L^p(0,T),\,g\in L^{q}(0,T)$ with $1\leq p,q\leq \infty$, and $\frac{1}{p}+\frac{1}{q}\leq 1$. Then, we have
 $$\|(fg)_t*w_{\varepsilon}-(f(g*w_{\varepsilon}))_t\|_{L^{r}(0,T)}\leq C\|f_t\|_{L^{p}(0,T)}\|g\|_{L^{q}(0,T)}$$
 for some $C\geq 0$ independent of $\varepsilon$, $f$ and $g$, $r$ is determined by $\frac{1}{r}=\frac{1}{p}+\frac{1}{q}.$ In addition,
  $$(fg)_t*w_{\varepsilon}-(f(g*w_{\varepsilon}))_t\to0\;\;\text{ in }\,L^{r}(0,T)$$
 as $\varepsilon \to 0$ if $r<\infty.$
 \end{Lemma}

 With Lemma \ref{Lions's lemma} and Lemma \ref{similar to Lions's lemma} in hand, we are ready to handle the terms $R_1$ and $R_2$.
For $\kappa>0,$ by Lemma \ref{lemma of estimates for positive kappa} and Poincar\'{e} inequality, we have $\v\in L^2(0,T;L^6(\O))$.
We also have, by Lemma \ref{lemma of estimates for positive kappa},
\begin{equation*}
\rho_t\in L^4(0,T;L^{6/5}(\O))+L^2(0,T;L^{3/2}(\O)).
\end{equation*}
Thus, applying Lemma \ref{similar to Lions's lemma},
\begin{equation}
\label{R1 goes to 0}
\begin{split}
|R_1|&\leq\int_0^T\int_{\O}\left|\psi(t)\varphi'_n(\overline{\v})[\overline{(\rho\v)_t}-(\rho\overline{\v})_t]\right|\,dx\,dt
\\&\leq C(\psi)\int_0^T\int_{\O}\left|\varphi'_n(\overline{\v})[\overline{(\rho\v)_t}-(\rho\overline{\v})_t]\right|\,dx\,dt\to0\;\;\;\;\text{as }\;k\to0.
\end{split}
\end{equation}

Similarly, applying Lemma \ref{Lions's lemma}, we conclude
\begin{equation}
\label{R2 goes to 0}
R_2\to0\;\;\;\text{ as }\;\;k\to0.
\end{equation}

By \eqref{convegence first term}-\eqref{R2 goes to 0}, we have
\begin{equation}
\label{the second level of MV inequality without initial term}
\begin{split}
-\int_0^T\int_{\O}\psi_t\rho\varphi_n(\v)\,dx\,dt&+\int_0^T\int_{\O}\psi(t)\varphi'_n(\v)F\,dx\,dt
\\&\quad\quad+\int_0^T\int_{\O}\psi(t)\mathbf{S}:\nabla(\varphi'_n(\v))\,dx\,dt
=0,
\end{split}
\end{equation}
for any test function $\psi\in \mathfrak{D}(0,\infty).$\\
Now, we need to consider the test function $\psi(t)\in \mathfrak{D}(-1,\infty).$ For this, we need the continuity of $\rho(t)$ and $(\sqrt{\rho}\u)(t)$ in the strong topology at $t=0$.

In fact, thanks to Proposition \ref{pro weak solutions}, we have $$(\sqrt{\rho})_t\in L^2(0,T;L^2(\O)),\quad\quad\sqrt{\rho}\in L^2(0,T;H^2(\O)).$$ This gives us
$$\sqrt{\rho}\in C([0,T];L^2(\O))\quad\quad\text{ and }\;\; \nabla\sqrt{\rho} \in C(0,T; L^2(\O)),$$
thanks to Theorem 3 on page 287, see \cite{Evans}.
Similarly, we have
\begin{equation}
\label{continuity for density}\rho\in C([0,T];L^2(\O))
\end{equation}
due to $$\|\nabla\rho\|_{L^2(0,T;L^2(\O))}\leq C\|\nabla\rho^{\frac{1}{4}}\|_{L^4(0,T;L^4(\O))}\|\rho^{\frac{3}{4}}\|_{L^4(0,T;L^4(\O))}.$$
Meanwhile, we have $$\sqrt{\rho}\in L^{\infty}(0,T;L^p(\O))\;\;\text{ for any } 1\leq p\leq 6,$$
and hence \begin{equation}
\label{continuity in strong topology 1}\sqrt{\rho}\in C([0,T];L^p(\O))\;\;\text{ for any } 1\leq p\leq 6.
\end{equation}
On the other hand, we see
\begin{equation}
\begin{split}
\label{control of sqrt density and u}
&\text{ess}\limsup_{t\to0}\int_{\O}|\sqrt{\rho}\u-\sqrt{\rho_0}\u_0|^2\,dx
\\&\leq\text{ess}\limsup_{t\to0}\left(\int_{\O}(\frac{1}{2}\rho|\u|^2+\frac{\rho^{\gamma}}{\gamma-1}+\kappa|\nabla\sqrt{\rho}|^2)\,dx-\int_{\O}(\frac{1}{2}\rho_0|\u_0|^2+\frac{\rho_0^{\gamma}}{\gamma-1})+\kappa|\nabla\sqrt{\rho}_0|^2\,dx\right)
\\&+\text{ess}\limsup_{t\to0}\left(2\int_{\O}\sqrt{\rho_0}\u_0(\sqrt{\rho_0}\u_0-\sqrt{\rho}\u)\,dx+\int_{\O}(\frac{\rho_0^{\gamma}}{\gamma-1}-\frac{\rho^{\gamma}}{\gamma-1})\right)
\\&-\kappa\,\text{ess}\limsup_{t\to0}\left|\nabla\sqrt{\rho}_0-\nabla\sqrt{\rho}\right|^2\,dx+2\kappa \,\text{ess} \limsup_{t\to 0}\int_{\O}\nabla\sqrt{\rho}_0\cdot(\nabla\sqrt{\rho}_0-\nabla\sqrt{\rho})\,dx.
\end{split}
\end{equation}
We have \begin{equation}
\label{grade density}
\text{ess} \limsup_{t\to 0}\int_{\O}\nabla\sqrt{\rho}_0\cdot(\nabla\sqrt{\rho}_0-\nabla\sqrt{\rho})\,dx=0.
\end{equation}

So, using \eqref{energy inequality for approximation}, \eqref{continuity in strong topology 1} and the convexity of $\rho\mapsto \rho^{\gamma}$, we have
\begin{equation*}
\begin{split}
&\text{ess}\limsup_{t\to0}\int_{\O}|\sqrt{\rho}\u-\sqrt{\rho_0}\u_0|^2\,dx
\\&\leq 2\text{ess}\limsup_{t\to0}\int_{\O}\sqrt{\rho_0}\u_0(\sqrt{\rho_0}\u_0-\sqrt{\rho}\u)\,dx\\&
=  2\text{ess}\limsup_{t\to0}\left(\int_{\O}\sqrt{\rho_0}\u_0(\sqrt{\rho_0}\u_0-\sqrt{\rho}\u\phi_m(\rho))\,dx+\int_{\O}\sqrt{\rho_0}\u_0(1-\phi_m(\rho))\sqrt{\rho}\u\,dx\right)
\\&=B_1+B_2.
\end{split}
\end{equation*}

By Proposition \ref{pro weak solutions}, we have
\begin{equation}
\label{weak continuous}
\rho\u\in C([0,T]; L^{\frac{3}{2}}_{\text{weak}}(\O)).
\end{equation}
We consider $B_1$ as follows
\begin{equation*}\begin{split}
&\quad\quad\quad B_1=\\&2\text{ess}\limsup_{t\to0}\left(\int_{\O}\sqrt{\rho_0}\u_0(\frac{\phi_m(\rho)}{\sqrt{\rho}}(\rho_0\u_0-\rho\u))\,dx-\int_{\O}\sqrt{\rho_0}\rho_0|\u_0|^2(\frac{\phi_m(\rho)}{\sqrt{\rho}}-\frac{\phi_m(\rho_0)}{\sqrt{\rho_0}})\,dx\right)
,\end{split}
\end{equation*}
then we have $B_1=0,$ where we used \eqref{continuity in strong topology 1}
and \eqref{weak continuous}.

Since $m\geq m_0$, and $\rho_0\geq \frac{1}{m_0}$,  we have
 \begin{equation*}
 |B_2|\leq \|\sqrt{\rho_0}\u_0\|_{L^{\infty}(0,T;\O)}\|\sqrt{\rho}\u\|_{L^{\infty}(0,T;L^2(\O))}\text{ess}\limsup_{t\to0}\|1-\phi_m(\rho)\|_{L^2(0,T;\O)}=0.
 \end{equation*}
 Thus, we have
\begin{equation*}
\text{ess}\limsup_{t\to0}\int_{\O}|\sqrt{\rho}\u-\sqrt{\rho_0}\u_0|^2\,dx=0,
\end{equation*}
which gives us
\begin{equation}
\label{continuity in strong topology 2}\sqrt{\rho}\u\in C([0,T];L^2(\O)).
\end{equation}
 By \eqref{continuity for density} and \eqref{continuity in strong topology 2}, we get
 $$\lim_{\tau\to0}\frac{1}{\tau}\int_0^{\tau}\int_{\O}\rho\varphi_n(\v)\,dx\,dt=\int_{\O}\rho_0\varphi_n(\v_0)\,dx.$$

 Considering \eqref{the second level of MV inequality without initial term} for the test function,
 \begin{equation*}
 \begin{split}
 &\psi_{\tau}(t)=\psi(t)\;\;\text{ for } \,t\geq \tau,\;\;\psi_{\tau}(t)=\psi(\tau)\frac{t}{\tau}\;\;\text{ for } t\leq \tau,
 \end{split}
 \end{equation*}
 we get
 \begin{equation*}
\begin{split}
-\int_{\tau}^T\int_{\O}\psi_t\rho\varphi_n(\v)\,dx\,dt&+\int_0^T\int_{\O}\psi_{\tau}(t)\varphi'_n(\v)F\,dx\,dt
\\&\quad+\int_0^T\int_{\O}\psi_{\tau}(t)\mathbf{S}:\nabla(\varphi'_n(\v))\,dx\,dt
=\frac{\psi(\tau)}{\tau}\int_0^{\tau}\int_{\O}\rho\varphi_n(\v)\,dx\,dt.
\end{split}
\end{equation*}
 Passing into the limit as $\tau\to 0$, this gives us
 \begin{equation}
\label{the second level of MV inequality before limit}
\begin{split}
-\int_0^T\int_{\O}\psi_t\rho\varphi_n(\v)\,dx\,dt&+\int_0^T\int_{\O}\psi(t)\varphi'_n(\v)F\,dx\,dt\\&+\int_0^T\int_{\O}\psi(t)\mathbf{S}:\nabla(\varphi'_n(\v))\,dx\,dt
=\int_{\O}\rho_0\psi(0)\varphi_n(\v_0)\,dx.
\end{split}
\end{equation}

\section{Recover the limits as $m\to\infty$}

In this section, we want to recover the limits in \eqref{the first level of MV inequality} as $m\to\infty.$ Here, we should remark that $(\rho,\u)$ is any fixed
 weak solution to \eqref{last level} verifying Proposition \ref{pro weak solutions} with $\kappa>0$. For any fixed weak solution $(\rho,\u)$, we have
$$\phi_m(\rho)\to1\;\;\;\text{ almost everywhere  for } (t,x),$$
and it is uniform bounded in $L^{\infty}(0,T;\O),$
we also have
$$r_0\phi_K(\rho)\u\in L^2(0,T;L^2(\O)),$$ and thus
$$\v_m=\phi_m\phi_K\u\to\phi_K\u\;\;\;\text{ almost everywhere for } (t,x)$$
as $m\to\infty.$ By the Dominated Convergence Theorem, we have $$\v_m\to\phi_K\u\;\;\;\text{ in } L^2(0,T;L^2(\O))$$
as $m\to\infty,$
 and hence, we have $$\varphi_n(\v_m)\to\varphi_n(\phi_K\u)\;\;\;\text{ in } L^p((0,T)\times\O)$$
 for any $1\leq p<\infty$.

 Meanwhile, for any fixed $\rho$, we have $$\phi'_m(\rho)\to0\;\;\;\text{ almost everywhere for } (t,x)$$
 as $m\to\infty.$
 We calculate $|\phi'_m(\rho)|\leq 2m$ as $\frac{1}{2m}\leq\rho\leq \frac{1}{m},$ and otherwise, $\phi'_m(\rho)=0,$ thus, we have
 $$|\rho\phi'_m(\rho)|\leq 1\;\;\;\text{ for all }\,\rho.$$

We can find that
$$\int_0^T\int_{\O}\psi'(t)(\rho\varphi_n(\v_m))\,dx\,dt\to \int_0^T\int_{\O}\psi'(t)(\rho\varphi_n(\phi_K(\rho)\u))\,dx\,dt$$
and $$\int_{\O}\rho_0\varphi_n(\v_{m0})\to \int_{\O}\rho_0\varphi_n(\phi_K(\rho_0)\u_0)$$
as $m\to\infty.$

To pass into the limits in \eqref{the second level of MV inequality before limit} as $m\to\infty,$  we rely on the following Lemma:
 \begin{Lemma}
 \label{lemma for the limit as m goes to infinity}
 If $$\|a_m\|_{L^{\infty}(0,T;\O)}\leq C,\;\; a_m\to a\;\;\text{ a.e. for } (t,x) \text{ and in } L^p((0,T)\times\O)\;\;\text{ for any } 1\leq p<\infty,$$
 $f\in L^1((0,T)\times\O)$, then we have
 $$\int_0^T\int_{\O}\phi_m(\rho)a_mf\,dx\,dt\to\int_0^T\int_{\O}af\,dx\,dt\;\;\;\text{ as } m\to\infty,$$
 and $$\int_0^T\int_{\O}\left|\rho\phi'_m(\rho)a_mf\right|\,dx\,dt\to 0\;\;\text{ as } m\to\infty.$$
 \end{Lemma}
\begin{proof}We have
\begin{equation*}
|\phi_m(\rho)a_mf-af|\leq|\phi_m(\rho)f-f||a_m|+|a_mf-af|=I_1+I_2.
\end{equation*}
For $I_1$:
$\phi_m(\rho)f\to f$ a.e. for $(t,x)$
and $$|\phi_m(\rho)f-f|\leq 2 |f|\;\;\text{ a.e. for } (t,x),$$
by Lebesgue's Dominated Convergence Theorem, we conclude
$$\int_0^T\int_{\O}\left|\phi_m(\rho)f-f\right|\,dx\,dt\to 0$$
as $m\to\infty,$
which in turn yields \begin{equation*}
\begin{split}&\int_0^T\int_{\O}|\phi_m(\rho)a_mf-a_mf|\,dx\,dt
\\&\leq\|a_m\|_{L^{\infty}(0,T;\O)}\int_0^T\int_{\O}|\phi_m(\rho)f-f|\,dx\,dt
\\&\to0
\end{split}
\end{equation*}
as $ m\to \infty.$
Following the same line, we have $$\int_0^T\int_{\O}|a_mf-af|\,dx\,dt\to0$$ as $m\to\infty.$ Thus we have
$$\int_0^T\int_{\O}\phi_m(\rho)a_mf\,dx\,dt\to\int_0^T\int_{\O}af\,dx\,dt$$
as $  m\to\infty.$\\
We now consider
$\int_0^T\int_{\O}|\rho\phi'_m(\rho)a_mf|\,dx\,dt$.
Notice that $|\rho\phi'_m(\rho)|\leq C$, and $\rho\phi_m'(\rho)$ converges to $0$ almost everywhere, so $\left|\rho\phi'_m(\rho)a_mf\right|\leq C|f|$, and by Lebesgue's Dominated Convergence Theorem,
 $$\int_0^T\int_{\O}|\rho\phi'_m(\rho)a_mf|\,dx\,dt\to 0$$
 as $ m\to\infty.$
\end{proof}
Calculating
\begin{equation}
\begin{split}
\label{Sm term}
&\int_0^T\int_{\O}\psi(t)\mathbf{S}_m:\nabla(\varphi'_n(\v_m))\,dx\,dt
\\&=\int_0^T\int_{\O}\psi(t)\mathbf{S}_m\varphi''_n(\v_m)\left(\nabla\phi_m\phi_K\u+\phi_m\nabla\phi_K\u+\phi_m\phi_K\nabla\u\right)\,dx\,dt\\
&=\int_0^T\int_{\O}\phi_m(\rho) a_{m1}f_{1}\,dx\,dt+\int_0^T\int_{\O}\rho\phi'_m(\rho) a_{m2}f_{2}\,dx\,dt,
\end{split}
\end{equation}
where $$a_{m1}=\phi_m(\rho)\varphi''_n(\v_m),$$
$$f_{1}=\psi(t)\rho\phi_K(\rho)\left(\mathbb{D}\u+\kappa\frac{\D\sqrt{\rho}}{\sqrt{\rho}}\mathbb{I}\right)(\u \nabla\phi_K(\rho)+\phi_K(\rho)\nabla\u),$$
and $$a_{m2}=\varphi''_n(\v_m)\phi_m(\rho)\phi_K(\rho)\u=\varphi''_n(\v_m)\v_m,$$
$$f_{2}=\psi(t)\phi_K(\rho)\left(\mathbb{D}\u+\kappa\frac{\D\sqrt{\rho}}{\sqrt{\rho}}\mathbb{I}\right)\nabla\rho=2\psi(t)\phi_K(\rho)(\kappa\Delta\sqrt{\rho}\nabla\sqrt{\rho}+\sqrt{\rho}\mathbb{D}\u\nabla\sqrt{\rho}).$$
So applying Lemma \ref{lemma for the limit as m goes to infinity} to \eqref{Sm term}, we have
$$\int_0^T\int_{\O}\psi(t)\mathbf{S}_m:\nabla(\varphi'_n(\v_m))\,dx\,dt\to \int_0^T\int_{\O}\psi(t)\mathbf{S}:\nabla(\varphi'_n(\phi_K(\rho)\u))\,dx\,dt$$
as $m\to\infty,$ where $\mathbf{S}=\phi_K(\rho)\rho(\mathbb{D}\u+\kappa\frac{\D\sqrt{\rho}}{\sqrt{\rho}}\mathbb{I}).$

Letting $F_m=F_{m1}+F_{m2},$
where \begin{equation*}
\begin{split}
F_{m1}&=\rho^2\u\phi'(\rho)\Dv\u
+\rho\nabla\phi(\rho)\mathbb{D}\u
+\kappa\sqrt{\rho}\nabla\phi(\rho)\Delta\sqrt{\rho}
\\&=\rho\left(\phi'_m(\rho)\phi_K(\rho)+\phi_m(\rho)\phi'_K(\rho)\right)(\rho\u\Dv\u+\nabla\rho\cdot\mathbb{D}\u+\kappa\nabla\rho\frac{\D\sqrt{\rho}}{\sqrt{\rho}}),
\end{split}
\end{equation*}
where $$\phi_K(\rho)(\rho\u\Dv\u+\nabla\rho\cdot\mathbb{D}\u+\kappa\nabla\rho\frac{\D\sqrt{\rho}}{\sqrt{\rho}})\in L^1((0,T)\times\O),$$
$$\rho\phi'_K(\rho)(\rho\u\Dv\u+\nabla\rho\cdot\mathbb{D}\u+\kappa\nabla\rho\frac{\D\sqrt{\rho}}{\sqrt{\rho}})\in L^1((0,T)\times\O),$$
and
\begin{equation*}
\begin{split}
F_{m2}&=\phi_m(\rho)\phi_K(\rho)(2\rho^{\frac{\gamma}{2}}\nabla\rho^{\frac{\gamma}{2}}
+r_0\u
+r_1\rho|\u|^2\u
+2\kappa\nabla\sqrt{\rho}\Delta\sqrt{\rho}),
\end{split}
\end{equation*}
where $$ \phi_K(\rho)\left(2\rho^{\frac{\gamma}{2}}\nabla\rho^{\frac{\gamma}{2}}
+r_0\u
+r_1\rho|\u|^2\u
+2\kappa\nabla\sqrt{\rho}\Delta\sqrt{\rho}\right)\in L^1((0,T)\times\O).$$

Applying Lemma \ref{lemma for the limit as m goes to infinity}, we have
\begin{equation*}
\int_0^T\int_{\O}\psi(t)\varphi'_n(\v_m)F_m\,dx\,dt\to \int_0^T\int_{\O}\psi(t)\varphi'_n(\phi_K(\rho)\u)F\,dx\,dt,
\end{equation*}
where\begin{equation*}
\begin{split}
F&=\rho^2\u\phi_K'(\rho)\Dv\u+2\rho^{\frac{\gamma}{2}}\nabla\rho^{\frac{\gamma}{2}}\phi_K(\rho)
+\rho\nabla\phi_K(\rho)\mathbb{D}\u+r_0\u\phi_K(\rho)
\\&+r_1\rho|\u|^2\u\phi_K(\rho)
+\kappa\sqrt{\rho}\nabla\phi_K(\rho)\Delta\sqrt{\rho}+2\kappa\phi_K(\rho)\nabla\sqrt{\rho}\Delta\sqrt{\rho}.
\end{split}
\end{equation*}
Letting $m\to\infty$ in \eqref{the second level of MV inequality before limit}, we have
\begin{equation*}
\begin{split}
&-\int_0^T\int_{\O}\psi'(t)(\rho\varphi_n(\phi_K(\rho)\u))\,dx\,dt+\int_0^T\int_{\O}\psi(t)\varphi'_n(\phi_K(\rho)\u)F\,dx\,dt
\\&+\int_0^T\int_{\O}\psi(t)\mathbf{S}:\nabla(\varphi'_n(\phi_K(\rho)\u))\,dx\,dt=\int_{\O}\psi(0)\rho_0\varphi_n(\phi_K(\rho_0)\u_0)\,dx,
\end{split}
\end{equation*}
which in turn gives us the following lemma:
\begin{Lemma}
For any weak solutions to \eqref{last level} verifying in Proposition \ref{pro weak solutions}, we have
\begin{equation}
\label{the second level of MV inequality after limit}
\begin{split}
&-\int_0^T\int_{\O}\psi'(t)(\rho\varphi_n(\phi_K(\rho)\u))\,dx\,dt+\int_0^T\int_{\O}\psi(t)\varphi'_n(\phi_K(\rho)\u)F\,dx\,dt
\\&+\int_0^T\int_{\O}\psi(t)\mathbf{S}:\nabla(\varphi'_n(\phi_K(\rho)\u))\,dx\,dt=\int_{\O}\psi(0)\rho_0\varphi_n(\phi_K(\rho_0)\u_0)\,dx,
\end{split}
\end{equation}
where $\mathbf{S}=\phi_K(\rho)\rho(\mathbb{D}\u+\kappa\frac{\Delta\sqrt{\rho}}{\sqrt{\rho}}\mathbb{I}),$ and
\begin{equation*}
\begin{split}
F&=\rho^2\u\phi_K'(\rho)\Dv\u+2\rho^{\frac{\gamma}{2}}\nabla\rho^{\frac{\gamma}{2}}\phi_K(\rho)
+\rho\nabla\phi_K(\rho)\mathbb{D}\u+r_0\u\phi_K(\rho)
\\&+r_1\rho|\u|^2\u\phi_K(\rho)
+\kappa\sqrt{\rho}\nabla\phi_K(\rho)\Delta\sqrt{\rho}+2\kappa\phi_K(\rho)\nabla\sqrt{\rho}\Delta\sqrt{\rho}.
\end{split}
\end{equation*}
where $\mathbb{I}$ is an identical matrix.
\end{Lemma}

\section{Recover the limits as $\kappa\to0$ and $K\to\infty$.}
The objective of this section is to recover the limits in \eqref{the second level of MV inequality after limit} as $\kappa\to0$ and $K\to\infty.$ In this section, we assume that $K=\kappa^{-\frac{3}{4}},$ thus $K\to\infty$ when $\kappa\to0.$  First, we address the following lemma.
\begin{Lemma}
\label{lemma on vanishing kappa}
Let $\kappa\to0$ and $K\to\infty$, and denote $\v_{\kappa}=\phi_K(\rho_{\kappa})\u_{\kappa}$, we have
\begin{equation}
\begin{split}
\label{density 5/3}&\rho^{\gamma}_{\kappa} \;\;\text{  is bounded in } L^r((0,T)\times\O)\;\text{ for any } 1\leq r<2 \;\text{ in } 2D,
\\&\text{and any } 1<r<\frac{5}{3} \;\text{ in }3D.
\end{split}
\end{equation}
\\ For any $g\in C^1(\R^+)$ with $g$ bounded, and $0<\alpha< \infty$ in 2D, $0<\alpha<\frac{5\gamma}{3}$ in 3D,
we have \begin{equation*}
\rho^{\alpha}_{\kappa}g(|\v_{\kappa}|^2)\to \rho^{\alpha} g(|\u|^2);\;\;\text{ strongly in } L^1((0,T)\times\O).
\end{equation*}
In particular, we have, for any fixed $n$,
\begin{equation}
\label{convergence 1 for kappa}
\rho_{\kappa}\varphi_n(\v_{\kappa})\to\rho\varphi_n(\u)\;\;\;\text{ strongly in } L^1((0,T)\times\O),
\end{equation}
and
\begin{equation}
\label{convergence 2 for kappa}
\rho^{2\gamma-1}_{\kappa}(1+\tilde{\varphi}'_n(|\v_{\kappa}|^2))\to
\rho^{2\gamma-1}(1+\tilde{\varphi}'_n(|\u|^2))\;\;\;\text{ strongly in } L^{1}((0,T)\times\O),
\end{equation}for $1<\gamma<3.$

\end{Lemma}
\begin{proof}

In 2D, we deduce that $$\rho^{\gamma}_{\kappa}\in L^{\infty}(0,T:L^1(\O))\cap L^1(0,T;L^p(\O))\;\;\text{ for any }1\leq p<\infty.$$ Thus, $\rho^{\gamma}_{\kappa}$ is bounded in
$L^r((0,T)\times\O)$ for $1\leq r<2.$

In 3D, we deduce that $$\rho^{\gamma}_{\kappa}\in L^{\infty}(0,T:L^1(\O))\cap L^1(0,T;L^3(\O)).$$
Applying Holder inequality, we have
$$\|\rho^{\gamma}_{\kappa}\|_{L^{\frac{5}{3}}((0,T)\times\O)}\leq \|\rho^{\gamma}_{\kappa}\|^{\frac{2}{5}}_{L^{\infty}(0,T;L^1(\O))} \|\rho^{\gamma}_{\kappa}\|^{\frac{3}{5}}_{L^{1}(0,T;L^3(\O))}.$$
 Thus, $\rho^{\gamma}_{\kappa}$ is bounded in
$L^{\frac{5}{3}}((0,T)\times\O)$.

 We have that $(\rho_{\kappa})_t$ is uniformly bounded in
 \begin{equation*}
L^4(0,T;L^{6/5}(\O))+L^2(0,T;L^{3/2}(\O)),
\end{equation*}
 thanks to Lemma \ref{lemma of estimates for positive kappa};
  and also we have
$$\|\nabla\rho_{\kappa}\|_{L^{\infty}(0,T;L^{3/2}(\O))}\leq C.$$
Applying Aubin-Lions Lemma,
one obtains $$\rho_{\kappa}\to \rho\;\;\;\text{ strongly in } L^{p}(0,T;L^{3/2}(\O))\quad\text{ for } p<\infty.$$
When $\kappa\to 0,$ we have $\sqrt{\rho_{\kappa}}\u_{\kappa}\to \sqrt{\rho}\u\;\;\text{strongly in } L^2(0,T;L^2(\O))$ in Proposition \ref{pro weak solutions}, ( also see \cite{VY-1}).
Thus, up to a subsequence,  for almost every $(t,x)$ such that $\rho(t,x)\neq0$,
we have $$\u_{\kappa}(t,x)=\frac{\sqrt{\rho_{\kappa}}\u_{\kappa}}{\sqrt{\rho_{\kappa}}}\to \u(t,x),$$
and $$\v_{\kappa}\to \u(t,x),$$
as $\kappa\to 0.$
For almost every $(t,x)$ such that $\rho(t,x)=0,$
\begin{equation}
\label{convergence kappa level}
\left|\rho^{\alpha}_{\kappa}g(|\v_{\kappa}|^2)\right|
\leq C\rho^{\alpha}_{\kappa}(t,x)\to 0=\rho^{\alpha}g(|\u|^2)
\end{equation}
as $\kappa\to0.$\\
Hence, $\rho^{\alpha}_{\kappa}g(|\v_{\kappa}|^2)$ converges to $\rho^{\alpha}g(|\u|^2)$ almost everywhere. Since $g$ is bounded and \eqref{density 5/3}, $\rho^{\alpha}_{\kappa} g(|\v_{\kappa}|^2)$ is uniformly bounded in
$L^r((0,T)\times\O)$ for some $r>1.$ Hence, $$\rho^{\alpha}_{\kappa}g(|\v_{\kappa}|^2)\to\rho^{\alpha}g(|\u|^2)\;\;\text{ in } L^1((0,T)\times\O).$$
By the uniqueness of the limit, the convergence holds for the whole sequence.

Applying this result with $\alpha=1$ and $g(|\v_{\kappa}|^2)=\varphi_n(\v_{\kappa})$, we deduce \eqref{convergence 1 for kappa}.

Since $\gamma>1$ in 2D, we can take $\alpha=2\gamma-1<2\gamma$; and take $\gamma<3$ in 3D, we have  $2\gamma-1<\frac{5\gamma}{3}$. Thus
 we use the above result with $\alpha=2\gamma-1$ and $g(|\v_{\kappa}|^2)=1+\tilde{\varphi}'_n(|\v_{\kappa}|^2)$ to obtain \eqref{convergence 2 for kappa}.

\end{proof}

 With the lemma in hand, we are ready to recover the limits
in \eqref{the second level of MV inequality after limit} as $\kappa\to0$ and $K\to\infty.$ We have the following lemma.
\begin{Lemma}
\label{Lemma kappa goes to 0}
Let $K=\kappa^{-\frac{3}{4}}$, and $\kappa\to 0,$ for any $\psi\geq 0$ and $\psi'\leq 0$, we have
\begin{equation}
\begin{split}
\label{MV before n goes to infinity}
&-\int_0^T\int_{\O}\psi'(t)\rho\varphi_n(\u)\,dx\,dt
\\&\leq
8\|\psi\|_{L^{\infty}}\left(
\int_{\O}\left(\rho_0|\u_0|^2+\frac{\rho^{\gamma}_0}{\gamma-1}+|\nabla\sqrt{\rho_0}|^2-r_0\log_{-}\rho_0\right)\,dx+2 E_0\right)
\\&+
C(\|\psi\|_{L^{\infty}})\int_0^T \int_{\O}(1+\tilde{\varphi}'_n(|\u|^2))\rho^{2\gamma-1} \,dx\,dt+\psi(0)\int_{\O}\rho_0\varphi_n(\u_0)\,dx
\end{split}
\end{equation}

\end{Lemma}
\begin{proof}Here, we use $(\rho_{\kappa},\u_{\kappa})$ to denote
  the weak solutions to \eqref{last level} verifying Proposition \ref{pro weak solutions} with $\kappa>0$.

 By Lemma \ref{lemma on vanishing kappa}, we can handle the first term in \eqref{the second level of MV inequality after limit}, that is,
\begin{equation}
\label{convergence of the first term in section 4}
\int_0^T\int_{\O}\psi'(t)(\rho_{\kappa}\varphi_n(\v_{\kappa}))\,dx\,dt\to
\int_0^T\int_{\O}\psi'(t)(\rho\varphi_n(\u))\,dx\,dt
\end{equation}
and \begin{equation}
\psi(0)\int_{\O}\rho_0\varphi'_n(\v_{\kappa,0})\,dx\to \psi(0)\int_{\O}\rho_0\varphi'_n(\u_0)\,dx
\end{equation}
as $\kappa\to 0$ and $K=\kappa^{-\frac{3}{4}}\to\infty.$

\begin{equation}
\begin{split}
&\int_0^T\int_{\O}\psi(t)\varphi'_n(\v_{\kappa})\cdot\nabla\rho_{\kappa}^{\gamma}\phi_K(\rho_{\kappa})\,dx\,dt=
\\&-\int_0^T\int_{\O}\psi(t)\rho_{\kappa}^{\gamma}\phi_K(\rho_{\kappa})\varphi_n^{''}:\nabla\v_{\kappa}\,dx\,dt-\int_0^T\int_{\O}\psi(t)\rho_{\kappa}^{\gamma}\varphi_n'(\v_{\kappa})\cdot\nabla\phi_K(\rho_{\kappa})\,dx\,dt
\\&=P_1+P_2.
\end{split}
\end{equation}
We can control $P_2$ as follows
\begin{equation}
\begin{split}
|P_2|&\leq \|\psi\|_{L^{\infty}}\int_0^T\int_{\O}|\rho_{\kappa}^{\gamma}||\varphi_n'||\nabla\phi_K(\rho_{\kappa})|\,dx\,dt
\\&\leq C(n,\|\psi\|_{L^{\infty}})\kappa^{-\frac{1}{4}}\|\phi_K'\sqrt{\rho_{\kappa}}\|_{L^{\infty}}\|\rho_{\kappa}^{\gamma+\frac{1}{4}}\|_{L^{\frac{4}{3}}(0,T;L^{\frac{4}{3}}(\O))}
\left(\kappa^{\frac{1}{4}}\|\nabla\rho_{\kappa}^{\frac{1}{4}}\|_{L^{4}(0,T;L^4(\O))}\right)
\\&\leq \frac{2}{\sqrt{K}} C(n,\|\psi\|_{L^{\infty}})\kappa^{-\frac{1}{4}}\|\rho_{\kappa}^{\gamma+\frac{1}{4}}\|_{L^{\frac{4}{3}}(0,T;L^{\frac{4}{3}}(\O))}\left(\kappa^{\frac{1}{4}}\|\nabla\rho_{\kappa}^{\frac{1}{4}}\|_{L^{4}(0,T;L^4(\O))}\right)
\\&\leq\frac{2C}{\sqrt{K}}\kappa^{-\frac{1}{4}}=2C\kappa^{\frac{1}{8}} \to 0
\end{split}
\end{equation}
as $\kappa\to 0,$ where we used  $\frac{4}{3}(\gamma+\frac{1}{4})\leq \frac{5\gamma}{3}$ for any $\gamma>1.$ \\
Calculating $P_1$,
\begin{equation}
\begin{split}
&P_1=-\int_0^T\int_{\O}\psi(t)\rho_{\kappa}^{\gamma}\phi_K(\rho_{\kappa})\varphi_n^{''}:\nabla\v_{\kappa}\,dx\,dt
\\&=-\int_0^T\int_{\O}\psi(t)\rho_{\kappa}^{\gamma}(\phi_K(\rho_{\kappa}))^2\varphi^{''}_n:\nabla\u_{\kappa}\,dx\,dt
\\&\quad\quad\quad
-\int_0^T\int_{\O}\psi(t)\rho_{\kappa}^{\gamma}\phi_K(\rho_{\kappa})\varphi^{''}_n:\left(\nabla\phi_K(\rho_{\kappa})\otimes\u_{\kappa}\right)\,dx\,dt
\\&=P_{11}+P_{12}.
\end{split}
\end{equation}
We bound $P_{12}$ as follows
\begin{equation}
\begin{split}
&|P_{12}|=-\int_0^T\int_{\O}\psi(t)\rho_{\kappa}^{\gamma}\phi_K(\rho_{\kappa})\varphi^{''}_n:\left(\nabla\phi_K(\rho_{\kappa})\otimes\u_{\kappa}\right)\,dx\,dt
\\&=-\int_0^T\int_{\O}\psi(t)\rho_{\kappa}^{\gamma}(\nabla\phi_K(\rho_{\kappa}))^{T}\varphi^{''}_n(\v_{\kappa})\v_{\kappa}\,dx\,dt
\\&\leq C(n,\|\psi\|_{L^{\infty}})\kappa^{-\frac{1}{4}}\|\rho_{\kappa}^{\gamma+\frac{1}{4}}\|_{L^{\frac{4}{3}}(0,T;L^{\frac{4}{3}}(\O))}\left(\kappa^{\frac{1}{4}}\|\nabla\rho^{\frac{1}{4}}_{\kappa}\|_{L^4(0,T;L^4(\O))}\right)\|\phi_K'\sqrt{\rho_{\kappa}}\|_{L^{\infty}}
\\&\leq\frac{2C}{\sqrt{K}}\kappa^{-\frac{1}{4}}=2C\kappa^{\frac{1}{8}} \to 0
\end{split}\end{equation}
as $\kappa\to0,$ where we used $|\phi_K'(\rho_{\kappa})\sqrt{\rho_{\kappa}}|\leq \frac{2}{\sqrt{K}},$ $\|\varphi^{''}_n(\v_k)\v_k\|_{L^{\infty}(0,T;L^{\infty}(\O))}\leq C(n)$ since $\varphi^{''}_n$ is compactly supported, and $\frac{4}{3}(\gamma+\frac{1}{4})\leq \frac{5\gamma}{3}$ for any $\gamma>1$. \\
Thanks to part b of Lemma \ref{Lemma for the cut function} and  $$\varphi^{''}_n(\v_{\kappa}):\nabla\u_{\kappa}=4\nabla\u_{\kappa}\tilde{\varphi}^{''}_n(|\v_{\kappa}|^2)\v_{\kappa}\otimes\v_{\kappa}+2\Dv\u_{\kappa}\tilde{\varphi}'_n(|\v_{\kappa}|^2),$$ we have
\begin{equation}
\begin{split}
|P_{11}|&\leq 4 \int_0^T\int_{\O}\psi(t)|\phi_K|^2|\rho_{\kappa}^{\gamma}||\nabla\u_{\kappa}|\,dx\,dt
+2\int_0^T\int_{\O}\psi(t)|\phi_K|^2|\tilde{\varphi}'_n(|\v_{\kappa}|^2)||\rho_{\kappa}^{\gamma}||\Dv\u_{\kappa}|\,dx\,dt
\\& \leq 4\|\psi\|_{L^{\infty}}\int_0^T\int_{\O}\rho_{\kappa}|\nabla\u_{\kappa}|^2\,dx\,dt+C(\|\psi\|_{L^{\infty}})\int_0^T\int_{\O}\rho_{\kappa}^{2\gamma-1}\,dx\,dt
\\&
+2\int_0^T\int_{\O}\psi(t)|\phi_K|^2|\tilde{\varphi}'_n(|\v_{\kappa}|^2)||\rho_{\kappa}^{\gamma}||\Dv\u_{\kappa}|\,dx\,dt,
\end{split}
\end{equation}
and the term
\begin{equation}
\begin{split}
&2\int_0^T\int_{\O}\psi(t)|\phi_K|^2|\tilde{\varphi}'_n(|\v_{\kappa}|^2)||\rho_{\kappa}^{\gamma}||\Dv\u_{\kappa}|\,dx\,dt
\\& \leq 2\int_0^T\int_{\O}\psi(t)|\phi_K|^2|\tilde{\varphi}'_n(|\v_{\kappa}|^2)|\rho_{\kappa}|\mathbb{D}\u_{\kappa}|^2 \,dx\,dt
\\&\quad\quad\quad\quad\quad+C(\|\psi\|_{L^{\infty}})\int_0^T \int_{\O}|\tilde{\varphi}'_n(|\v_{\kappa}|^2)|\rho_{\kappa}^{2\gamma-1} \,dx\,dt.
\end{split}
\end{equation}
Thus,
\begin{equation}
\begin{split}
\label{P11}
|P_{11}|\leq 4\|\psi\|_{L^{\infty}}&\int_0^T\int_{\O}\rho_{\kappa}|\nabla\u_{\kappa}|^2\,dx\,dt
\\&+2\int_0^T\int_{\O}\psi(t)|\phi_K|^2|\tilde{\varphi}'_n(|\v_{\kappa}|^2)|\rho_{\kappa}|\mathbb{D}\u_{\kappa}|^2 \,dx\,dt
\\&\quad\quad\quad+
C(\|\psi\|_{L^{\infty}})\int_0^T \int_{\O}(1+\tilde{\varphi}'_n(|\v_{\kappa}|^2))\rho_{\kappa}^{2\gamma-1} \,dx\,dt.
\end{split}
\end{equation}
The first right hand side term will be controlled by $$4\|\psi\|_{L^{\infty}}\left(\int_{\O}\left(\rho_0|\u_0|^2+\frac{\rho^{\gamma}_0}{\gamma-1}+|\nabla\sqrt{\rho_0}|^2-r_0\log_{-}\rho_0\right)\,dx+2 E_0\right)$$
due to \eqref{estimate on u};
and the second right hand side term will be absorbed by the dispersion term $A_1$ in \eqref{A1 absorb term}.
By Lemma \ref{lemma on vanishing kappa}, we have
\begin{equation}
\label{the last term}
\begin{split}
&\int_0^T \int_{\O}(1+\tilde{\varphi}'_n(|\v_{\kappa}|^2)\rho_{\kappa}^{2\gamma-1} \,dx\,dt\to
\int_0^T \int_{\O}(1+\tilde{\varphi}'_n(|\u|^2))\rho^{2\gamma-1} \,dx\,dt
\end{split}
\end{equation}
as $\kappa\to 0.$

Note that
\begin{equation}
\int_0^T\int_{\O}\psi(t)\varphi'_n(\v_{\kappa})(r_0\u_{\kappa}+r_1\rho_{\kappa}|\u_{\kappa}|^2\u_{\kappa})\,dx\,dt\geq 0,
\end{equation}
  so this term can be dropped directly.

We treat the other terms  in $F$ one by one,
\begin{equation}
\begin{split}
&\int_0^T\int_{\O}\left|\psi(t)\varphi'_n(\v_{\kappa})\rho_{\kappa}^2\u_{\kappa}\phi_K'(\rho_{\kappa})\Dv\u_{\kappa}\right|\,dx\,dt
\\&\leq C(n,\psi)\|\rho_{\kappa}^{\frac{1}{4}}\u_{\kappa}\|_{L^4((0,T;L^4(\O))}\|\sqrt{\rho_{\kappa}}\Dv\u_{\kappa}\|_{L^2((0,T;L^2(\O))}
\\&\times\|\phi'_K(\rho_{\kappa})\sqrt{\rho_{\kappa}}\|_{L^{\infty}}\|\rho_{\kappa}^{\frac{3}{4}}\|_{L^4((0,T;L^4(\O))}
\leq C(n,\psi)\kappa^{\frac{3}{8}}\to 0
\end{split}
\end{equation}
as $\kappa\to 0$, where we used Sobolev inequality,
and $|\phi'_K(\rho_{\kappa})\sqrt{\rho_{\kappa}}|\leq \frac{2}{\sqrt{K}};$
\begin{equation}
\begin{split}
&\int_0^T\int_{\O}\left|\psi(t)\varphi'_n(\v_{\kappa})\rho_{\kappa}\nabla\phi_K(\rho_{\kappa})\mathbb{D}\u_{\kappa}\right|\,dx\,dt
\\&\leq C(n,\psi)\frac{|\phi'_K(\rho_{\kappa})\sqrt{\rho_{\kappa}}|}{\kappa^{\frac{1}{4}}}\left(\kappa^{\frac{1}{4}}\|\nabla\rho_{\kappa}^{\frac{1}{4}}\|_{L^4((0,T;L^4(\O))}\right)\|\sqrt{\rho_{\kappa}}\mathbb{D}\u_{\kappa}\|_{L^2((0,T;L^2(\O))}\|\rho_{\kappa}^{\frac{3}{4}}\|_{L^4((0,T;L^4(\O))}
\\&\leq C(n,\psi)\kappa^{\frac{1}{8}}\to 0
\end{split}
\end{equation}
as $\kappa\to0;$
\begin{equation}
\begin{split}
&\kappa\int_0^T\int_{\O}\left|\psi(t)\varphi'_n(\v_{\kappa})\sqrt{\rho_{\kappa}}\nabla\phi_K(\rho_{\kappa})\Delta\sqrt{\rho_{\kappa}}\right|\,dx\,dt
\\&\leq 2C(n,\psi)\kappa^{\frac{1}{4}}\left(\kappa^{\frac{1}{4}}\|\nabla\rho_{\kappa}^{\frac{1}{4}}\|_{L^4((0,T;L^4(\O))}\right)\|\sqrt{\kappa}\Delta\sqrt{\rho_{\kappa}}\|_{L^2((0,T;L^2(\O))}\|\rho_{\kappa}^{\frac{1}{4}}\|_{L^4(0,T;L^4(\O))}
\\&\leq 2C(n,\psi)\kappa^{\frac{1}{4}}
\to 0
\end{split}
\end{equation}
as $\kappa\to0,$ where we used $|\rho_{\kappa}\phi'_K(\rho_{\kappa})|\leq 1.$ Finally
\begin{equation}
\begin{split}
&\kappa\int_0^T\int_{\O}\left|\psi'(t)\varphi_n(\v_{\kappa})\phi_K(\rho_{\kappa})\nabla\sqrt{\rho_{\kappa}}\Delta\sqrt{\rho_{\kappa}}\right|\,dx\,dt
\\&\leq 2C(n,\psi)\kappa^{\frac{1}{4}}\left(\kappa^{\frac{1}{4}}\|\nabla\rho_{\kappa}^{\frac{1}{4}}\|_{L^4((0,T;L^4(\O))}\right)\|\sqrt{\kappa}\Delta\sqrt{\rho_{\kappa}}\|_{L^2((0,T;L^2(\O))}\|\rho_{\kappa}^{\frac{1}{4}}\|_{L^4(0,T;L^4(\O))}
\\&\leq 2C(n,\psi)\kappa^{\frac{1}{4}}
\to 0
\end{split}
\end{equation}
as $\kappa\to0.$\\
For the term $\mathbf{S}_{\kappa}=\phi_K(\rho_{\kappa})\rho_{\kappa}(\mathbb{D}\u_{\kappa}+\kappa\frac{\Delta\sqrt{\rho_{\kappa}}}{\sqrt{\rho_{\kappa}}}\mathbb{I})=\mathbf{S}_1+\mathbf{S}_2,$ we calculate as follows
\begin{equation}
\begin{split}
&\int_0^T\int_{\O}\psi(t)\mathbf{S}_1:\nabla(\varphi'_n(\v_{\kappa}))\,dx\,dt
\\&=
\int_0^T\int_{\O}\psi(t)\phi_K(\rho_{\kappa})\rho_{\kappa}\mathbb{D}\u_{\kappa}:\nabla(\varphi'_n(\v_{\kappa}))\,dx\,dt
\\&=
\int_0^T\int_{\O}\psi(t)[\nabla\u_{\kappa}\varphi''_n(\v_{\kappa})\rho_{\kappa}]:\mathbb{D}\u_{\kappa}(\phi_K(\rho_{\kappa}))^2\,dx\,dt\\&
+\int_0^T\int_{\O}\psi(t)\rho_{\kappa} \phi_K(\rho_{\kappa})(\u_{\kappa}^T\varphi_n^{''}(\v_{\kappa}))\mathbb{D}\u_{\kappa}\nabla(\phi_K(\rho_{\kappa}))\,dx\,dt
\\&=A_1+A_2.
\end{split}
\end{equation}
For $A_1$, by part a. of Lemma \ref{Lemma for the cut function}, we have
\begin{equation}
\label{inequality A1}
\begin{split}
A_1&=
\int_0^T\int_{\O}\psi(t)[\nabla\u_{\kappa}\varphi''_n(\v_{\kappa})\rho_{\kappa}]:\mathbb{D}\u_{\kappa}(\phi_K(\rho_{\kappa}))^2\,dx\,dt
\\&=2\int_0^T\int_{\O}\psi(t)\tilde{\varphi}'_n(|\v_{\kappa}|^2)(\phi_K(\rho_{\kappa}))^2\rho_{\kappa}\mathbb{D}\u_{\kappa}:\nabla\u_{\kappa}\,dx\,dt
\\&+4
\int_0^T\int_{\O}\psi(t)\rho_{\kappa}(\phi_K(\rho_{\kappa}))^2\tilde{\varphi}^{''}_n(|\v_{\kappa}|^2)(\nabla\u_{\kappa}\v_{\kappa}\otimes\v_{\kappa}):\mathbb{D}\u_{\kappa}\,dx\,dt
\\&=A_{11}+A_{12}.
\end{split}
\end{equation}
Notice that $$\mathbb{D}\u_{\kappa}:\nabla\u_{\kappa}=|\mathbb{D}\u_{\kappa}|^2,$$
thus
\begin{equation}
\begin{split}
\label{A1 absorb term}
A_1&\geq 2\int_0^T\int_{\O}\psi(t)\tilde{\varphi}'_n(|\v_{\kappa}|^2)(\phi_K(\rho_{\kappa}))^2\rho_{\kappa}|\mathbb{D}\u_{\kappa}|^2\,dx\,dt
\\&-4\|\psi\|_{L^{\infty}}\int_0^T\int_{\O}\rho_{\kappa}|\nabla\u_{\kappa}|^2\,dx\,dt,
\end{split}
\end{equation}
where we control $A_{12}$
\begin{equation*}\begin{split}
A_{12}&\leq 4\int_0^T\int_{\O}|\psi(t)|\frac{|\v_{\kappa}|^2}{1+|\v_{\kappa}|^2}\rho_{\kappa}|\nabla\u_{\kappa}|^2\,dx\,dt
\leq
 4\|\psi\|_{L^{\infty}}\int_0^T\int_{\O}\rho_{\kappa}|\nabla\u_{\kappa}|^2\,dx\,dt
\\& \leq
4\|\psi\|_{L^{\infty}}\left(
\int_{\O}\left(\rho_0|\u_0|^2+\frac{\rho^{\gamma}_0}{\gamma-1}+|\nabla\sqrt{\rho_0}|^2-r_0\log_{-}\rho_0\right)\,dx+2 E_0\right),
 \end{split}
\end{equation*}thanks to \eqref{estimate on u}.
For $A_2$, thanks to \eqref{bound of second deritative of cut function}, we can control it as follows
\begin{equation}\begin{split}
|A_2|&\leq C(n,\psi)\|\sqrt{\rho_{\kappa}}\mathbb{D}\u_{\kappa}\|_{L^2((0,T;L^2(\O))}\|\rho_{\kappa}^{\frac{3}{4}}\|_{L^4((0,T;L^4(\O))}
\\&\times(\kappa^{\frac{1}{4}}\|\nabla\rho_{\kappa}^{\frac{1}{4}}\|_{L^4((0,T;L^4(\O))})\frac{\|\phi'_K\sqrt{\rho_{\kappa}}\|_{L^{\infty}((0,T)\times\O)}}{\kappa^{\frac{1}{4}}}
\\&\leq \frac{C}{\sqrt{K}\kappa^{\frac{1}{4}}}=C\kappa^{\frac{1}{8}}\to0
\end{split}
\end{equation}
as $\kappa\to 0$. Note that the first right hand side term of \eqref{A1 absorb term} has a positive sign and control the limit using from the pressure \eqref{P11}.
We need to treat the term related to $\mathbf{S}_2$,
\begin{equation}
\begin{split}
&\kappa\int_0^T\int_{\O}\psi(t)\mathbf{S}_2:\nabla(\varphi'_n(\v_{\kappa}))\,dx\,dt
\\&=
\kappa\int_0^T\int_{\O}\psi(t)\nabla\u_{\kappa}\varphi''_n(\v_{\kappa}):\sqrt{\rho_{\kappa}}\phi_K(\rho_{\kappa})^2\Delta\sqrt{\rho_{\kappa}}\,dx\,dt\\&
+\kappa\int_0^T\int_{\O}\psi(t)\u_{\kappa}\phi_K(\rho_{\kappa})\varphi''_n(\v_{\kappa})\nabla\rho_{\kappa}\sqrt{\rho_{\kappa}}\phi'_K(\rho_{\kappa})\Delta\sqrt{\rho_{\kappa}}\,dx\,dt
\\&=B_1+B_2,
\end{split}
\end{equation}
we control $B_1$ as follows
\begin{equation}
\begin{split}
|B_1|&\leq C(n,\psi)\|\sqrt{\rho_{\kappa}}\nabla\u_{\kappa}\|_{L^2(0,T;L^2(\O))}\|\sqrt{\kappa}\Delta\sqrt{\rho_{\kappa}}\|_{L^2(0,T;L^2(\O))}\sqrt{\kappa}
\\&\leq C\kappa^{\frac{1}{2}}\to 0
\end{split}
\end{equation}
as $\kappa\to 0,$ where we used $\|\phi_K^2\|_{L^{\infty}(0,T;L^{\infty}(\O))}\leq C$.\\
For $B_2$, we have
\begin{equation}
\begin{split}
\label{B2 goes to 0}
|B_2|&\leq
C(n,\psi)\kappa^{\frac{1}{4}}\left(\kappa^{\frac{1}{4}}\|\nabla\rho_{\kappa}^{\frac{1}{4}}\|_{L^4(0,T;L^4(\O))}\right)
\\&\times\|\rho_{\kappa}^{\frac{1}{4}}\|_{L^4(0,T;L^4(\O))}\|\sqrt{\kappa}\Delta\sqrt{\rho_{\kappa}}\|_{L^2(0,T;L^2(\O))}
\|\phi'_K(\rho_{\kappa})\rho_{\kappa}\|_{L^{\infty}(0,T;L^{\infty}(\O))}
\\&\leq C\kappa^{\frac{1}{4}}\to0
\end{split}
\end{equation}
as $\kappa\to 0.$

With \eqref{convergence of the first term in section 4}-\eqref{B2 goes to 0}, in particularly,  letting $\kappa\to 0$ in  \eqref{the second level of MV inequality after limit}, dropping the positive terms on the left side, we have the following inequality
\begin{equation*}
\begin{split}
-\int_0^T\int_{\O}&\psi'(t)\rho\varphi_n(\u)\,dx\,dt
\\&\leq
8\|\psi\|_{L^{\infty}}\left(
\int_{\O}\left(\rho_0|\u_0|^2+\frac{\rho^{\gamma}_0}{\gamma-1}+|\nabla\sqrt{\rho_0}|^2-r_0\log_{-}\rho_0\right)\,dx+2 E_0\right)
\\&+\psi(0)\int_{\O}\rho_0\varphi_n(\u_0)\,dx
+
C(\|\psi\|_{L^{\infty}})\int_0^T \int_{\O}(1+\tilde{\varphi}'_n(|\u|^2))\rho^{2\gamma-1} \,dx\,dt,
\end{split}
\end{equation*}
which in turn gives us Lemma \ref{Lemma kappa goes to 0}.
\end{proof}

\section{Recover the limits as $n\to\infty$.} In this section, we aim at recovering the Mellet-Vasseur type inequality for the weak solutions at the approximation level of compressible Navier-Stokes equations by letting $n\to\infty.$ In particular, we prove Theorem \ref{MV inequality} by recovering the limit from Lemma \ref{Lemma kappa goes to 0}. In this section, $(\rho,\u)$ are the fixed weak solutions.

Our task is to bound the right term of \eqref{MV before n goes to infinity},
\begin{equation}
\label{the last term}
\begin{split}
&C(\|\psi\|_{L^{\infty}})\int_0^T \int_{\O}(1+\tilde{\varphi}'_n(|\u|^2))\rho^{2\gamma-1} \,dx\,dt
\\&\leq C(\|\psi\|_{L^{\infty}})\int_0^T \left(\int_{\O}(\rho^{2\gamma-1-\frac{\delta}{2}})^{\frac{2}{2-\delta}}\,dx\right)^{\frac{2-\delta}{2}}\left(\int_{\O}\rho(1+\tilde{\varphi}'_n(|\u|^2))^{\frac{2}{\delta}}\,dx\right)^{\frac{\delta}{2}}\,dt\\&
\leq C(\|\psi\|_{L^{\infty}})\int_0^T \left(\int_{\O}(\rho^{2\gamma-1-\frac{\delta}{2}})^{\frac{2}{2-\delta}}\right)^{\frac{2-\delta}{2}}\left(\int_{\O}\rho(2+\ln(1+|\u|^2)))^{\frac{2}{\delta}}\,dx\right)^{\frac{\delta}{2}}\,dt,
\end{split}
\end{equation}
where we used part c of Lemma \ref{Lemma for the cut function}.
By \eqref{MV before n goes to infinity} and \eqref{the last term}, we have
\begin{equation*}
\begin{split}
&-\int_0^T\int_{\O}\psi'(t)\rho\varphi_n(\u)\,dx\,dt
\leq \int_{\O}\rho_0\varphi_n(\u_0)\,dx
\\&+8\|\psi\|_{L^{\infty}}\left(
\int_{\O}\left(\rho_0|\u_0|^2+\frac{\rho^{\gamma}_0}{\gamma-1}+|\nabla\sqrt{\rho_0}|^2-r_0\log_{-}\rho_0\right)\,dx+2 E_0\right)
\\&+C(\|\psi\|_{L^{\infty}})\int_0^T \left(\int_{\O}(\rho^{2\gamma-1-\frac{\delta}{2}})^{\frac{2}{2-\delta}}\right)^{\frac{2-\delta}{2}}\left(\int_{\O}\rho(2+\ln(1+|\u|^2)))^{\frac{2}{\delta}}\,dx\right)^{\frac{\delta}{2}}\,dt.
\end{split}
\end{equation*}
Thanks to part e. of Lemma \ref{Lemma for the cut function} and Monotone Convergence Theorem, we have
\begin{equation}
-\int_0^T\int_{\O}\psi'(t)\rho\varphi_n(\u)\,dx\,dt\to -\int_0^T\int_{\O}\psi'(t)\rho(1+|\u|^2)\ln(1+|\u|^2)\,dx\,dt
\end{equation}
as $n\to\infty.$

Letting $n\to\infty$, we have
\begin{equation}
\begin{split}
\label{last level approximation for MV}
&-\int_0^T\int_{\O}\psi'(t)\rho(1+|\u|^2)\ln(1+|\u|^2)\,dx\,dt
\leq \psi(0)\int_{\O}\rho_0(1+|\u_0|^2)\ln(1+|\u_0|^2)\,dx
\\&+8\|\psi\|_{L^{\infty}}\left(
\int_{\O}\left(\rho_0|\u_0|^2+\frac{\rho^{\gamma}_0}{\gamma-1}+|\nabla\sqrt{\rho_0}|^2-r_0\log_{-}\rho_0\right)\,dx+2 E_0\right)
\\&
+C\int_0^T \left(\int_{\O}(\rho^{2\gamma-1-\frac{\delta}{2}})^{\frac{2}{2-\delta}}\right)^{\frac{2-\delta}{2}}\left(\int_{\O}\rho(2+\ln(1+|\u|^2))^{\frac{2}{\delta}}\,dx\right)^{\frac{\delta}{2}}\,dt.
\end{split}
\end{equation}
Taking \begin{equation*}
\psi(t)\begin{cases}=1 \;\;\;\;\;\quad\quad\quad\quad\text{ if }t\leq \tilde{t}-\frac{\epsilon}{2}
\\ =\frac{1}{2}-\frac{t-\tilde{t}}{\epsilon}\;\;\;\;\;\;\;\;\;\;\text{ if } \tilde{t}-\frac{\epsilon}{2}\leq t\leq \tilde{t}+\frac{\epsilon}{2}
\\ =0\,\;\;\;\;\quad\quad\;\quad\quad\text{ if }  t\geq \tilde{t}+\frac{\epsilon}{2},
\end{cases}\end{equation*}
then \eqref{last level approximation for MV} gives for every $\tilde{t}\geq \frac{\epsilon}{2},$
\begin{equation*}
\begin{split}&
\frac{1}{\epsilon}\int_{\tilde{t}-\frac{\epsilon}{2}}^{\tilde{t}+\frac{\epsilon}{2}}\left(\int_{\O}\rho(1+|\u|^2)\ln(1+|\u|^2)\,dx\right)\,dt
\\&\leq\int_{\O}\rho_0(1+|\u_0|^2)\ln(1+|\u_0|^2)\,dx+8\left(
\int_{\O}\left(\rho_0|\u_0|^2+\frac{\rho^{\gamma}_0}{\gamma-1}+|\nabla\sqrt{\rho_0}|^2-r_0\log_{-}\rho_0\right)\,dx+2 E_0\right)
\\&
+C\int_0^T \left(\int_{\O}(\rho^{2\gamma-1-\frac{\delta}{2}})^{\frac{2}{2-\delta}}\right)^{\frac{2-\delta}{2}}\left(\int_{\O}\rho(2+\ln(1+|\u|^2))^{\frac{2}{\delta}}\,dx\right)^{\frac{\delta}{2}}\,dt.
\end{split}
\end{equation*}
This gives Theorem \ref{MV inequality} thanks to the Lebesgue point Theorem.

\section{Recover the weak solutions}
The objective of this section is to apply Theorem \ref{MV inequality} to prove Theorem \ref{main result}. In particular, we aim at establishing the existence of global weak solutions to \eqref{NS equation}-\eqref{initial data} by letting $r_0\to 0$ and $r_1\to 0.$ Let $r=r_0=r_1$, we use $(\rho_r,\u_r)$ to denote the weak solutions to \eqref{last level} verifying Proposition \ref{pro weak solutions} with $\kappa=0$.

By \eqref{energy inequality for approximation} and \eqref{BD entropy for approximation},  one obtains the following estimates,
\begin{equation}
\label{all priori estimates}
\begin{split}
&\|\sqrt{\rho_r}\u_r\|_{L^{\infty}(0,T;L^2(\O))}\leq C;\\
&\|\rho_r\|_{L^{\infty}(0,T;L^1\cap L^{\gamma}(\O))}\leq C;\\
&\|\sqrt{\rho_r}\nabla\u_r\|_{L^2(0,T;L^2(\O))}\leq C;\\
&\|\nabla\sqrt{\rho_r}\|_{L^{\infty}(0,T;L^2(\O))}\leq C;\\
&\|\nabla\rho_r^{\gamma/2}\|_{L^2(0,T;L^2(\O))}\leq C.
\end{split}
\end{equation}
 Theorem \ref{MV inequality} gives us
\begin{equation}
\begin{split}
\label{MV estimate}
\sup_{t\in[0,T]}\int_{\O}\rho_r|\u_r|^2\ln(1+|\u_r|^2)\,dx\leq C.
\end{split}
\end{equation}
It is necessary to remark that all above estimates on \eqref{all priori estimates} and \eqref{MV estimate} are uniformly on $r$.  Thus, we can make use of all estimates to recover the weak solutions by letting $r\to0$. Meanwhile, we have the following estimates from \eqref{energy inequality for approximation},
\begin{equation}
\label{r0 r1 estimates}
\begin{split}
&\int_0^T\int_{\O}r|\u_r|^{2}\,dx\,dt\leq C,\\
&\int_0^T\int_{\O}r\rho_r|\u_r|^4\,dx\,dt\leq C.
\end{split}
\end{equation}

To establish the existence of global weak solutions, we should pass to the limits as $r\to 0$. Following the same line as in \cite{MV}, we can show the convergence of the density and the pressure, prove the strong convergence of $\sqrt{\rho_r}\u_r$ in space $L^2_{loc}((0,T)\times\O)$, the convergence of the
diffusion terms. We remark that Theorem \ref{MV inequality} is the key tool to show the strong convergence of $\sqrt{\rho_r}\u_r.$ Here, we list all related convergence from \cite{MV}. In particular,
\begin{equation}
\label{convergence of density}
\begin{split}
&\sqrt{\rho_r}\to\ \sqrt{\rho}\;\;\text{ almost everywhere and strongly in } L^2_{loc}((0,T)\times\O)),\\
& \rho_r\to \rho \;\;\text{ in }\; C^0(0,T;L^{\frac{3}{2}}_{loc}(\O));
\end{split}
\end{equation}
the convergence of pressure
\begin{equation}
\begin{split}
\label{convegence of pressure}
\rho_r^{\gamma}\to \rho^{\gamma}\;\;\text{ strongly in }\; L^1_{loc}((0,T)\times\O);
\end{split}
\end{equation}
the convergence of the momentum and $\sqrt{\rho_r}\u_r$
\begin{equation}
\label{convergence of the momentum}
\begin{split}
&\rho_r\u_r\to \rho\u \;\;\text{ strongly in }\; L^2(0,T;L^p_{loc}(\O)) \;\;\text { for }\; p\in [1,3/2);\\
&\sqrt{\rho_r}\u_r\to\sqrt{\rho}\u\;\;\text{ strongly in }\; L^2_{loc}((0,T)\times\O);
\end{split}
\end{equation}
and the convergence of the diffusion terms
\begin{equation}
\label{convergence of diffusion terms}
\begin{split}
&\rho_r\nabla\u_r\to\rho\nabla\u\;\;\text{ in } \;\mathfrak{D}^{'},
\\&\rho_r\nabla^{T}\u_r\to\rho\nabla^{T}\u\;\;\text{ in }\;\mathfrak{D}^{'}.
\end{split}
\end{equation}

It remains to prove that terms $r\u_r$ and $r\rho_r|\u_r|^2\u_r$ tend to zero as $r\to0.$
Let $\psi$ be any test function, then we estimate the term $r\u_r$
\begin{equation}
\label{convergence of r0}
\begin{split}
&\left|\int_0^T\int_{\O}r\u_{r}\psi\;dx\,dt\right|\leq \int_0^T\int_{\O}r^{\frac{1}{2}}r^{\frac{1}{2}}|\u_r||\psi|\;dx\,dt
\\&\leq\sqrt{r}\|\sqrt{r}\u_r\|_{L^2((0,T)\times\O)}\|\psi\|_{L^2((0,T)\times\O)}\to 0
\end{split}
\end{equation}
as $r\to 0$,
due to \eqref{r0 r1 estimates}.

We also estimate $r\rho_r|\u_r|^2\u_r$ as follows
\begin{equation}
\begin{split}
\label{convergence of r1}
&\left|\int_0^T\int_{\O}r\rho_r|\u_r|^2\u_r\psi\,dx\,dt\right|
\\&\leq\sqrt{r}\|\sqrt{r}\sqrt{\rho_r}|\u_r|^2\|_{L^2((0,T)\times\O)}\|\sqrt{\rho_r}\u_r\|_{L^{\infty}(0,T;L^2(\O))}\|\psi\|_{L^{\infty}((0,T)\times\O)}\to 0
\end{split}
\end{equation}
as $r\to0.$

The global weak solutions to \eqref{last level} verifying Proposition \ref{pro weak solutions} with $\kappa=0$ is in the following sense, that is, $(\rho_r,\u_r)$
satisfy the following weak formulation
\begin{equation}
\begin{split}
\label{weak formulation before passing limit}
&\int_{\O}\rho_r\u_r\cdot\psi\,dx|_{t=0}^{t=T}-\int_{0}^{T}\int_{\O}\rho_r\u_r\psi_t\,dx\,dt
-\int_{0}^{T}\int_{\O}\rho_r\u_r\otimes\u_r:\nabla \psi\,dx\,dt\\&-\int_{0}^{T}\int_{\O}\rho_r^{\gamma}\Dv\psi\,dx\,dt
-\int_0^T\int_{\O}\rho\mathbb{D}\u_r:\nabla\psi\,dx\,dt
\\& =-r\int_{0}^{T}\int_{\O}\u_r\psi\,dx\,dt-r\int_0^T\int_{\O}\rho_r|\u_r|^2\u_r\psi\,dx\,dt,
\end{split}
\end{equation}
where
 $\psi$ is any test function.

Letting $r\to0$ in the weak formulation \eqref{weak formulation before passing limit}, and applying \eqref{convergence of density}-\eqref{convergence of r1}, one obtains that
\begin{equation}
\begin{split}
\label{weak formulation}
&\int_{\O}\rho\u\cdot\psi\,dx|_{t=0}^{t=T}-\int_{0}^{T}\int_{\O}\rho\u\psi_t\,dx\,dt
-\int_{0}^{T}\int_{\O}\rho\u\otimes\u:\nabla \psi\,dx\,dt\\&-\int_{0}^{T}\int_{\O}\rho^{\gamma}\Dv\psi\,dx\,dt
-\int_0^T\int_{\O}\rho\mathbb{D}\u:\nabla\psi\,dx\,dt=0.
\end{split}
\end{equation}
Thus we proved Theorem \ref{main result} for any initial value $(\rho_0,\u_0)$ verifying \eqref{initial condition} with the additional condition $\rho_0\geq\frac{1}{m_0}.$ This last condition can be dropped using \cite{MV}.

\section{acknowledgement}
A. Vasseur's research was supported in part by NSF grant DMS-1209420. C. Yu's reserch was
supported in part by an AMS-Simons Travel Grant.

\end{document}